\documentclass[12pt,a4paper]{article}
\usepackage[english]{babel}

\usepackage{pstricks}
\usepackage{pst-node}
\usepackage{amsmath}
\usepackage{amssymb}
\usepackage{graphicx}
\usepackage{enumerate}
\usepackage{color}
\usepackage[text={15cm,24cm}]{geometry}

\newenvironment{proof}{{\bf Proof}:\ }%
{~\ \hfill $\Box$\vspace{0,5cm}}

{~\ \hfill$\Diamond$\vspace{0,5cm}}

\newtheorem{prop}{Property}[section]
\newtheorem{theorem}{Theorem}[section]

\newtheorem{lemma}[theorem]{Lemma}

\newtheorem{coro}[theorem]{Corollary}

\newrgbcolor{lightlightlightgray}{0.9 0.9 0.9}

\numberwithin{equation}{section}

\begin{document}

\title{The Minimum Dominating Set problem is polynomial for $(claw,P_8)$-free graphs}
\author{
Valentin Bouquet\footnotemark[1] \footnotemark[2]
\and
Christophe\ Picouleau  \footnotemark[2]
}
\date{\today}

\def\thefootnote{\fnsymbol{footnote}}

\footnotetext[1]{ \noindent
Corresponding author: {\tt valentin.bouquet@cnam.fr}}

\footnotetext[2]{ \noindent
Conservatoire National des Arts et M\'etiers, CEDRIC laboratory, Paris (France). Email: {\tt
valentin.bouquet@cnam.fr,chp@cnam.fr}
}

\graphicspath{{.}{graphics/}}

%\begin{document}
\maketitle
\begin{abstract}

We prove that the Minimum Dominating Set problem is polynomial for the class of $(claw,P_8)$-free graphs.

 \vspace{0.2cm}
 \end{abstract}

%\newpage
\parindent=0cm
%----------------------------------------------------------------------------------------------------------------------------------
\section{Introduction}
M. Yannakakis and F. Gavril \cite{Yannakakis} showed in 1980 that the Minimum Dominating Set problem restricted to claw-free graphs is $NP$-complete. Then in 1984, A. Bertossi \cite{Bertossi} showed that the Minimum Dominating Set problem is also $NP$-complete for split graphs, a subclass of $P_5$-free graphs. More recently, in 2016, D. Malyshev \cite{K14P5} proved that the Minimum Dominating Set problem is polynomial for $(K_{1,4},P_5)$-free graphs, hence for $(claw,P_5)$-free graphs. To our knowledge, the complexity of the Minimum Dominating Set problem is unknown for $(claw,P_k)$-free graphs for every fixed $k\ge 6$. We show that the Minimum Dominating Set problem is polynomial for $(claw,P_8)$-free graphs.

\subsection*{Definitions and notations}
We are only concerned with simple undirected graphs $G=(V,E)$. The reader is referred to \cite{Bondy} for definitions and notations in graph theory. For $v\in V$,   $N(v)$ denotes its neighborhood and $N[v]=N(v)\cup\{v\}$ its closed neighborhood. A vertex $v$ is {\it universal} if $N[v]=V$. For $v\in V$ and $A\subseteq V,$ we denote by $N_A(v)=N(v)\cap A$ ($N_A[v]=(N(v)\cap A)\cup\{v\}$) its (closed) neighborhood in $A$. For $X\subseteq V$, $A\subseteq V,$ we denote $N_A(X)=\bigcup_{x\in X} N_A(x)$ and $N_A[X]=N_A(X)\cup X$.

The \emph{contraction} of an edge $uv\in E$ removes the vertices $u$ and $v$ from $V$, and replaces them by a new vertex that is adjacent to the previous neighbors of $u$ and $v$ (neither introducing self-loops nor multiple edges). The graph obtained from $G$ after the contraction of $uv$ is denoted by $G/uv$.

For $S\subseteq V$, let $G[S]$ denote the subgraph of $G$ {\it induced} by $S$, which has vertex set~$S$ and edge set $\{uv\in E\; |\; u,v\in S\}$. For $v\in V$, we write $G-v=G[V\setminus \{v\}]$ and for a subset $V'\subseteq V$ we write $G-V'=G[V\setminus V']$. For a fixed graph $H$ we write $H\subseteq_i G$ whenever $G$ contains an induced subgraph isomorphic to $H$.
For a set $\{H_1,\ldots,H_p\}$ of graphs, $G$ is {\it $(H_1,\ldots,H_p)$-free} if $G$ has no induced subgraph isomorphic to a graph in $\{H_1,\ldots,H_p\}$; if $p=1$ we may write $H_1$-free instead of $(H_1)$-free. For two disjoint induced subgraphs $G[A],G[B]$ of $G$, $G[A]$ is {\it complete} to $G[B]$ if $ab\in E$ for every $a\in A,b\in B$, $G[A]$ is {\it anticomplete} to $G[B]$ if $ab\not\in E$ for every $a\in A,b\in B$.

For $k\geq 1$, $P_k=u_1-u_2-\cdots-u_k$ is the {\it cordless path} on $k$ vertices, that is, $V({P_k})=\{u_1,\ldots,u_k\}$ and $E(P_k)=\{u_iu_{i+1}\; |\; 1\leq i\leq k-1\}$.
For $k\geq 3$, $C_k=u_1-u_2-\cdots-u_k-u_1$ is the {\it cordless cycle} on $k$ vertices, that is, $V({C_k})=\{u_1,\ldots,u_k\}$ and $E({C_k})=\{u_iu_{i+1}\; |\; 1\leq i\leq k-1\}\cup \{u_ku_1\}$. For $k\ge 4$, $C_k$ is called a {\it hole}. A graph without a hole is {\it chordal}.

A set $S\subseteq V$ is called a {\it stable set} or an {\it independent set} if $G[S]$ does not contain any edge. The maximum cardinality of an independent set in $G$ is denoted by $\alpha(G)$. A set $S\subseteq V$ is called a {\it clique} if $G[V]$ is a {\it complete graph}, i.e., every pairwise distinct vertices $u,v\in S$ are adjacent. The graph $C_3=K_3$ is a {\it triangle}. $K_{1,p}$ is the star on $p+1$ vertices, that is, the graph with vertices $u,v_1,v_2\ldots,v_p$ and edges $uv_1,uv_2,\cdots,uv_p$. The {\it claw} is $K_{1,3}$.

A set $S\subseteq V$ is a {\it dominating set} if every vertex $v\in V$ is either an element of $S$ or is adjacent to an element of $S$. The minimum cardinality of a dominating set in $G$ is denoted by $\gamma(G)$ and called the {\it domination number} of $G$. A dominating set $S$ with $\vert S\vert=\gamma(G)$ is called a {\it minimum dominating set}. Following \cite{DomBook} a minimum dominating set is also called a $\gamma$-set. We denote $V^+\subseteq V$ the subset of vertices $v$ of $G$ such that $\gamma(G-v)>\gamma(G)$. If $S\subset V$ is both a dominating and an independent set then $S$ is an {\it independent dominating set}. The minimum cardinality of an independent dominating set in $G$ is denoted by $i(G)$. Clearly we have $\gamma(G)\le i(G)\le \alpha(G)$. Note that a minimum independent dominating set is a {\it minimum maximal independent set}.

\subsection*{Previous results}
We give some results of the literature concerning the Minimum Dominating Set problem that will be useful in the following.
D. Bauer et al. showed in \cite{DomAlter} that for every non-isolated vertex $v$, if  $v\in V^+$ then $v$ is in every $\gamma$-set of $G$. Allan et al. \cite{IndDom} proved that $\gamma(G)=i(G)$ holds for every claw-free graph.
Yannakakis et al. \cite{Yannakakis} proved that the Minimum Dominating Set problem restricted to claw-free graphs is $NP$-complete.  D. Malyshev \cite{K14P5} proved that the Minimum Dominating Set problem is polynomial for $(K_{1,4},P_5)$-free graphs hence for $(claw,P_5)$-free graphs. As Farber \cite{IndDomChordal} proved, a minimum independent dominating set can be determined in linear-time over the class of chordal graphs, the Minimum Dominating Set problem restricted to claw-free chordal graphs is polynomial.

\subsection*{Organization}
The next section give some algorithmic properties. Two properties will allow us to make some simplifications on the graphs $G$ that we consider. Two others will help us to conclude that computing $\gamma(G)$ is polynomial when $G$ have a specific structure relatively to a fixed size subgraph. Then we consider the case where the graph $G$ has a long cycle. From there, we show our main result, starting from $(claw,P_6)$-free graphs and finishing with $(claw,P_8)$-free graphs. We conclude by some open questions regarding $(claw,P_k)$-free graphs for $k\ge 9$.

\section{Algorithmic Properties}
We give two properties that authorize us to make some assumptions and simplifications for the graphs we consider.
\begin{prop}\label{contract}
Let $G$ be a graph. If $u,v$ are two vertices such that $N[u]=N[v]$ then $\gamma(G)=\gamma(G/uv)$. \end{prop}
\begin{proof}
Let $u'$ be the vertex of $G/uv$ resulting from the contraction of $uv$. Let $\Gamma$ be a $\gamma$-set of $G$. At most one of $u$ and $v$ is in $\Gamma$. If $u\in \Gamma$ then let $\Gamma'=(\Gamma\setminus\{u\})\cup\{u'\}$. If $u,v\not\in\Gamma$ then let $\Gamma'=\Gamma$. In the two cases $\Gamma'$  is a dominating set of $G/uv$, so $\gamma(G)=\vert \Gamma'\vert\ge \gamma(G/uv)$.
Now suppose that $\gamma(G)> \gamma(G/uv)$. Let $\Gamma'$ be a $\gamma$-set of $G/uv$. If $u'\in \Gamma'$ then $(\Gamma'\setminus\{u'\})\cup\{u\}$ is a dominating set of $G$ such that $\vert (\Gamma'\setminus\{u'\})\cup\{u\}\vert=\gamma(G/uv)<\gamma(G)$, a contradiction. If $u'\not\in \Gamma'$ then $\Gamma'$ is a dominating set of $G$, a contradiction. Hence $\gamma(G)=\gamma(G/uv)$.
\end{proof}

\begin{prop}\label{leaf}
Let $G=(V,E)$ be a connected $claw$-free graph with $uv\in E$ such that $u$ is a leaf.  There exists $\Gamma$ a minimum dominating set of $G$ that consists of $\Gamma=\{v\}\cup\Gamma'$ where $\Gamma'$ is a minimum dominating set of $G'=G-N[v]$.
\end{prop}

\begin{proof}
Since $u$ is a leaf there exists $\Gamma$ a minimum dominating set of $G$ with $v\in\Gamma$. Let $w\in N(v)\setminus \{u\}$. Since $G$ is claw-free then $N(w)\setminus N[v]$ is a clique. We can assume that $w\not\in \Gamma$, otherwise replacing $w$ by $w'\in N(w)\setminus N[v]$ we have another $\gamma$-set of $G$ (note that if $N(w)\setminus N[v]$ is empty then $\Gamma$ cannot be a minimum dominating set). We show that $\Gamma'=\Gamma\setminus N[v]$ is a minimum dominating set of $G'=G-N[v]$. Clearly $\Gamma'$ dominates $G'$. If there exists $S$ a $\gamma$-set of $G'$ such that $\vert S\vert <\vert \Gamma'\vert$ then $S\cup\{v\}$ is a  dominating set of $G$ with $\vert S\cup\{v\}\vert<\Gamma$, a contradiction.
\end{proof}

As a consequence if a minimum dominating set of $G'=G-N[v]$ can be determined in polynomial time then  a minimum dominating set of $G$ can be determined in polynomial time.\\

We show two conditions on the structure of $G$ that authorize us to directly conclude that computing a $\gamma$-set for $G$ can be done in polynomial time.

\begin{prop}\label{V=N[T]}
Let $k>0$ be a fixed positive integer and $G=(V,E)$ be a graph. If there exists $T \subset V$ of size $\vert T\vert\leq k$ such that $V=N[T]$ then computing a minimum dominating set for $G$ is polynomial.
\end{prop}
\begin{proof}
We have $\gamma(G) \leq k$. So a minimum dominating set can be computed in $O(n^k)$.
\end{proof}

\begin{prop}\label{V=N[T]+W}
Let $k, k'>0$ two fixed positive integers and $G=(V,E)$ be a graph. If there exists $T \subset V$ of size $\vert T\vert\leq k$ such that $W=V\setminus N[T]$ has a size $\vert W\vert\leq k'$ then computing a minimum dominating set for $G$ is polynomial.
\end{prop}
\begin{proof}
We have $\gamma(G) \leq k + k'$. So a minimum dominating set can be computed in $O(n^{k+k'})$.
\end{proof}

\section{$G$ has a long cycle}
We give two lemmas that will authorize us to conclude that the Minimum Dominating Set problem is polynomial when $G$, a $(claw,P_k)$-free graph, contains a long induced cycle.

\begin{lemma}\label{Ck}
For every fixed $k\ge 6$, if $G$ is a $(claw,P_k)$-free connected graph such that $C_k\subseteq_i G$, then a minimum dominating set of $G$ can be given in polynomial time.
\end{lemma}
\begin{proof}
Let $C_k=v_1-\cdots-v_k-v_1,C_k\subseteq_i G$. Let $v\not\in V(C_k)$ be such that $N(v)\cap V(C_k)\ne \emptyset$. Since $G$ is claw-free and $k\ge 6$, we have $2\le \vert N(v)\cap V(C_k)\vert\le 4$. If $\vert N(v)\cap V(C_k)\vert=2$, the two neighbors of $v$ in $C_k$ must be adjacent, thus there is an induced $P_k$-subgraph that is a contradiction. For $3\le \vert N(v)\cap V(C_k)\vert\le 4$, let $w$ be a neighbor of $v$. If $N(w)\cap V(C_k)=\emptyset$ then there is a claw centered onto $v$, a contradiction. Hence every neighbor of $v$ has a neighbor in $C_k$ and therefore $N[C_k]=V$. So, from Property \ref{V=N[T]} we can compute a $\gamma$-set of $G$ in polynomial time.
\end{proof}

\begin{lemma}\label{Ck-1}
For every fixed $k\ge 6$, if $G$ is a $(claw,P_k,C_k)$-free connected graph such that $C_{k-1}\subseteq_i G$, then a minimum dominating set of $G$ can be given in polynomial time.
\end{lemma}
\begin{proof}
Let $C_{k-1}=v_1-\cdots-v_{k-1}-v_1$, $C_{k-1}\subseteq_i G$ and $v\not\in V(C_{k-1})$ such that $N(v)\cap V(C_{k-1})\ne \emptyset$. We have $2\le \vert N(v)\cap V(C_{k-1})\vert\le 4$ for $k\ge 7$ and $2\le \vert N(v)\cap V(C_{k-1})\vert\le 5$ for $k=6$.
Let $w$ be a neighbor of $v$ such that $N(w)\cap V(C_{k-1})=\emptyset$.
If $3\le \vert N(v)\cap V(C_{k-1})\vert\le 5$, then there is a claw centered onto $v$, a contradiction. When $\vert N(v)\cap V(C_{k-1})\vert=2$ there is an induced $P_k$-subgraph that is a contradiction.
 So $N[C_{k-1}]=V$ and therefore from Property \ref{V=N[T]} we can compute a $\gamma$-set of $G$ in polynomial time.
\end{proof}

\section{$G$ is $(claw,P_k,C_k, C_{k-1})$-free, $C_{k-2}\subseteq_i G$, $k\le 8$}

In this section we prove that, for $k\le 8$, if $G$ is a $(claw,P_k,C_k, C_{k-1})$-free graph such that $C_{k-2}\subseteq_i G$ then the Minimum Dominating Set problem is polynomial. The first lemma gives a structural property for $G$. We use this property to prove two  other lemmas, the first one for $k= 6$, the second for $7\le k\le 8$.

\begin{lemma}\label{Ck-2W}
For every fixed $k \geq 6$, if $G$ is a $(claw,P_k,C_k, C_{k-1})$-free connected graph such that $C_{k-2}\subseteq_i G$, then $W=V\setminus N[V(C_{k-2})]$ is an independent set.
\end{lemma}
\begin{proof}
Let $C=C_{k-2}=v_1-\cdots-v_{k-2}-v_1$,  $C\subseteq_i G$ and $v \in N[V(C)]\setminus V(C)$. We have $2 \le \vert N_C(v) \vert \le 5$ (note that $\vert N_C(v) \vert = 5$ only for $C=C_5$). Let $W=V\setminus N[V(C)]$ and let $w \in W$ be a neighbor of $v$.
If $3\le \vert N_C(v)\vert\le 5$, there is a claw, a contradiction. Hence, $v$ is such that $N_C(v)=\{v_i, v_{i+1}\}$, $1 \leq i \leq k-2$ (for convenience, when $i=k-2$, we read $v_{i + 1} = v_1$).
By Property \ref{contract}, we can assume that all contractibles vertices of $G$ are contracted. Moreover, from Property \ref{leaf} we can assume that $G$ has no leaves.

Assume for contradiction that $w$ has a neighbor $w'$, $w' \in W$. When $w'$ has no neighbor in  $N(V(C))$, there is an induced $P_k$-subgraph that is a contradiction. Hence $w'$ has a neighbor in $N(V(C))$.
Recall that $N[w] \neq N[w']$. If $vw' \not\in E$ then there is an induced $P_k$-subgraph, a contradiction. Hence, $w$ and $w'$ have the same neighbors in $N(V(C))$ but not in $W$. So there exists $r \in W$ with $rw \in E,rw' \not\in E$. The arguments above implies $rv\in E$. But $G[\{r, v, v_{i}, w'\}]$ is a claw, a contradiction. Hence, $W=V\setminus N[V(C_{k-2})]$ is independent.
\end{proof}

\begin{lemma}\label{C56-2}
If $G$ is a $(claw,P_6,C_6,C_5)$-free connected graph such that $C_4\subseteq_i G$, then a minimum dominating set of $G$ can be given in polynomial time.
\end{lemma}

\begin{proof}
Let $C=C_4=v_1-\cdots-v_4-v_1$, $C\subseteq_i G$ and $v\not\in V(C)$ such that $N(v)\cap V(C)\ne \emptyset$. We have $2 \le \vert N_C(v) \vert \le 4$. Let $W=V\setminus N[C]$ and $w \in W$ be a neighbor of $v$. If $3\le \vert N_C(v)\vert\le 4$ then $G$ contains a claw, a contradiction. Hence, $N_C(v)=\{v_i, v_{i+1}\}$, $1 \leq i \leq 4$ (for convenience, when $i=4$, we read $v_{i + 1} = v_1$). We  assume that all contractibles vertices of $G$ are contracted and $G$ has no leaves.

By Property \ref{V=N[T]+W}, if $\vert W \vert \leq 1$ then a minimum dominating set can be computed in polynomial time. So we assume that $\vert W \vert \geq 2$ and by Lemma \ref{Ck-2W}, we know that $W$ is an independent set. We show that all vertices $v \in N[W]\setminus W$ have exactly the same neighbors in $C$.

Let $w, w' \in W$, $w \neq w'$, be such that $w$ has a neighbor $v\in N[C]\setminus V(C)$ and $w'$ has a neighbor $v'\in N[C]\setminus V(C)$. Since $G$ is claw-free $v \neq v'$. W.l.o.g. $N_C(v)=\{v_1, v_2\}$.
Assume that $N_C(v) \neq N_C(v')$. W.l.o.g. $N_C(v')=\{v_2, v_3\}$ (note that $N_C(v')=\{v_1, v_4\}$ is symmetric). If $vv' \not\in E$ then $w-v-v_1-v_4-v_3-v'=P_6$, else $v_1-v-v'-v_3-v_4=C_5$, a contradiction. Now it remains $N_C(v')=\{v_3, v_4\}$. We have $vv'\not\in E$ else there is a claw, but $w-v-v_1-v_4-v'-w'=P_6$, a contradiction. Thus, w.l.o.g. every vertex $w\in W$ has only neighbors $v \in N[C]\setminus V(C)$ such that $N(v)=\{v_1, v_2\}$.

Let $\vert W \vert = q$, $q \geq 2$. We show that $\gamma(G)=q+1$.
Since $W$ is independent and for every distinct $w, w' \in W$, we have $N[w] \cap N[w'] = \emptyset$, we must take $q$ vertices of $N[W]$ to dominate the vertices of $W$. This vertices cannot dominate $v_3$ nor $v_4$. Hence $\gamma(G)\ge q+1$.

We construct a $\gamma$-set of $G$ as follows. We set $R$ by taking exactly one neighbor of each $w, w \in W$. Clearly, $\Gamma=R \cup \{v_3\}$ dominates $V(C) \cup N[R]$. Suppose that there exists $s \in N[C]\setminus V(C)$ that is not dominated by $\Gamma$. If $N_C(s)=\{v_1, v_2\}$ then there exists $r \in R$ such that $G[\{r, s, v_1, v_4\}]$ is a claw, a contradiction. If $N_C(s)=\{v_1, v_4\}$ then $w-v-v_2-v_3-v_4-s=P_6$, a contradiction. If $N_C(s)=\{v_1, v_2, v_4\}$ then there exists $r \in R$ such that $G[\{r, s, v_2, v_3\}]$ is a claw, a contradiction. Hence every $s \not\in N[R] \cup V(C)$ is dominated by $v_3$. It follows that $\Gamma$ is a $\gamma$-set of $G$.
Clearly $\Gamma$ can be constructed in polynomial time.\end{proof}

\begin{lemma}\label{C78-2}
For $k\in \{7, 8\}$, if $G$ is a $(claw,P_k,C_k, C_{k-1})$-free connected graph such that $C_{k-2}\subseteq_i G$, then a minimum dominating set of $G$ can be given in polynomial time.
\end{lemma}

\begin{proof}
By Properties \ref{contract} and \ref{leaf}, we can assume that all contractibles vertices of $G$ are contracted and that $G$ has no leaves.
Let $C=C_{k-2}=v_1-\cdots-v_{k-2}-v_1, C\subseteq_i G$ and $v \in N[C]\setminus V(C)$. We have $2 \le \vert N_C(v) \vert \le 5$ (note that $\vert N_C(v) \vert = 5$ only for $C=C_5$). Let $S=N[C]\setminus V(C)$, $W=V\setminus N[C]$ and $w \in W$ a neighbor of $v$. If $3\le \vert N_C(v)\vert\le 4$ then $G$ has a claw, a contradiction. Hence, $v$ is such that $N_C(v)=\{v_i, v_{i+1}\}$, $1 \leq i \leq k-2$ (for convenience, when $i=k-2$, we read $v_{i + 1} = v_1$). \\

We show that for every $w \in W,$ there exists $v, v' \in N(w)$ such that $N_C(v) \cap N_C(v') = \emptyset$. Let $w \in W$ and $v, v' \in N_S(w)$, $v \neq v'$.

First, we show that $N_C(v) \neq N_C(v')$. Suppose that $N_C(v)=N_C(v')$, w.l.o.g. $N_C(v)=\{v_1, v_2\}$. We have $vv'\in E$ else $G[\{v, v', v_1,v_{k-2}\}]$ is a claw. Since $N[v]\neq N[v']$ there exists $u \in V$ such that $uv \in E$ and $uv' \not\in E$. If $u \in W$ then by Lemma \ref{Ck-2W} $uw\not\in E$ but $G[\{u, v, w, v_1\}]$ is a claw, a contradiction. So, we have $u \in S$. If $N_C(u)=\{v_1, v_2\}$ then $G[\{u, v', v_2, v_3\}$ is a claw, a contradiction. So $N_C(u) \neq N_C(v)$ and we can assume that $wu\not\in E$, otherwise we have $u,v$ two neighbors of $w$ with distinct neighborhoods in $C$. If $N_C(u) \cap N_C(v) = \emptyset$ then $G[\{u, v, v_1, w\}]$ is a claw, a contradiction. So, w.l.o.g., we assume that $N_C(u) \cap N_C(v) = \{v_1\}$ but $G[\{u, v, v_2, w\}]$ is a claw, a contradiction. Hence $N[v] = N[v']$ and $v, v'$ can be contracted implying that  $w$ is a leaf, a contradiction. Thus for every $w,w \in W,$ there exists $v, v' \in N_S(w)$, $v\neq v'$ such that $N_C(v) \neq N_C(v')$.

Now we show that that $N_C(v) \cap N_C(v') = \emptyset$. W.l.o.g. assume that $N_C(v)=\{v_1, v_2\}$ and $N_C(v')=\{v_2, v_3\}$. If $vv' \in E$ then $v_1-v-v'-v_3-\cdots-v_{k-2}-v_1=C_{k-1}$, else $v_1-v-w-v'-v_3-\cdots-v_{k-2}-v_1=C_k$, a contradiction. Thus every $w \in W,$ has two neighbors $v,v'\in S$ such that $N_C(v) \cap N_C(v') = \emptyset$. \\

It follows from Property \ref{V=N[T]+W} that we can assume that $\vert W \vert \geq 2$. So let $w, w' \in W$ (recall $ww'\notin E$). Since both $w$ and $w'$ have two neighbors in $S$ with non intersecting neighborhoods in $C$, let $v \in N(w)$, $v' \in N(w')$ such that $N_C(v)\cap N_C(v')=\emptyset$. W.l.o.g. $N_C(v)=\{v_1, v_2\}$. Assume that $N_C(v')=\{v_3, v_4\}$ (note that $N(v')=\{v_{k-2}, v_{k-3}\}$ is symmetric). If $vv' \in E$ then $G[\{v, v', v_1, w\}]$ is a claw, else $w-v-v_1-v_{k-2}-\cdots-v_4-v'-w'=P_k$, a contradiction. Hence the two neighborhoods of $N_C(v)$ and $N_C(v')$ are not adjacent. It follows that for $k=7$, since $C_{k-2}=C_5$, such a configuration is impossible. This yields to $\vert W \vert \leq 1$ and by Property \ref{V=N[T]+W} a minimum dominating set can be computed in polynomial time. \\

Now, we focus on the remaining case of $k=8$. Let $\vert W \vert=q,q\ge 2$. We show that $\gamma(G)=q+2$. Since $W$ is independent and that for every distinct vertices $w, w' \in W$, we have $N[w] \cap N[w'] = \emptyset$, we must take $q$ vertices of $N[W]$ to dominate the vertices of $W$. Let $w, w' \in W$. From above we can assume that $w$ has a neighbor $v$ such that $N_C(v)=\{v_1,v_2\}$ and $w'$ has a neighbor $v'$ such that $N_C(v')=\{v_4,v_5\}$ (each vertex of $W$ has two neighbors whose are neighbors of respectively $\{v_1,v_2\}$ and $\{v_4,v_5\}$ since $C=C_6$). $G$ being claw-free we have $vv'\not\in E$. The $q$ vertices that dominates $W$ cannot dominate $v_3$ and $v_6$. Hence $\gamma(G)\ge q+1$.

Suppose that $\gamma(G)= q+1$. The minimum dominating set of $G$ must contain a vertex $s \in S$ a neighbor of both $v_3$ and $v_6$. If $vs\in E$, respectively $v's\in E$, then $G$ has a claw ($s$ cannot be complete to $N_C(v)\cup N_c(v')$), a contradiction. Also, $s$ must have ($v_1$ or $v_5$) and ($v_2$ or $v_4$) as neighbors else there is a claw. We assume first that $N(s)=\{v_1, v_2, v_3, v_6\}$. Then $w-v-v_1-s-v_3-v_4-v'-w'=P_8$ (recall $vv' \not\in E$ since $G$ is claw-free), a contradiction. The case where $N(s)=\{v_3, v_4, v_5, v_6\}$ is symmetric. Now we assume that $N(s)=\{v_1, v_3, v_4, v_6\}$ (note that $N(s)=\{v_2, v_3, v_5, v_6\}$ is symmetric). Then $w-v-v_2-v_3-s-v_6-v_5-v'=P_8$, a contradiction. Hence $\gamma(G)\ge q+2$. \\

We show that $\Gamma=\{v_1, v_4\} \cup W$ is a $\gamma$-set of $G$. Clearly $\Gamma$ dominates $N[W]\cup V(C)$. Let $s\not\in N[W]\cup V(C)$. So $s \in S$. Suppose that $sv_1, sv_4 \not\in E$. From above $ws\not\in E$ and $vs\not\in E$ else $G[\{v,s,v_1,w\}]$ is a claw. If $N(s)=\{v_2, v_3\}$ then $w-v-v_1-v_6-v_5-v_4-v_3-s=P_8$, a contradiction. By symmetry $N(s) \neq \{v_5, v_6\}$. As shown before $N(s)=\{v_2, v_3, v_5, v_6\}$ is not possible.
 Hence every $s\not\in N[W]\cup V(C)$ is dominated by $v_1$ or $v_4$. It follows that $\Gamma=\{v_1, v_4\} \cup W$ is a $\gamma$-set of $G$.
\end{proof}

By Lemmas \ref{Ck}, \ref{Ck-1}, \ref{C56-2}, \ref{C78-2} we immediately obtain the corollary below.
\begin{coro}\label{Ck,k-1,k-2}
Let $G$ a $(claw, P_k)$-free graph, $6 \leq k \leq 8$. If $C_l \subseteq_i G$, $k-2 \leq l \leq k$, then a minimum dominating set of $G$ can be given in polynomial time.
\end{coro}

\section{$G$ is $(claw,P_8)$-free}
Here we conclude by the main result proving that the Minimum Dominating Set problem is polynomial in the class of $(claw,P_8)$-free graphs. Starting from the result stating that the problem is polynomial when $G$ is $(claw,P_5)$-free, we successively prove that the problem is polynomial for $(claw,P_6)$-free, $(claw,P_7)$-free graphs. Then we conclude for the class of $(claw,P_8)$-free graphs.

In \cite{K14P5} D. Malyshev proved that the Minimum Dominating Set problem is polynomial for the class of $(K_{1,4},P_5)$-free graphs. Hence we obtain the following lemma.
\begin{lemma}\label{clawp5}
Let $G$ be a connected $(claw,P_5)$-free graph. Computing a minimum dominating set is polynomial-time solvable.
\end{lemma}

\begin{lemma}\label{clawp6}
Let $G$ be a connected $(claw,P_6)$-free graph. Computing a minimum dominating set is polynomial-time solvable.
\end{lemma}

\begin{proof}
It follows from  Corollary \ref{Ck,k-1,k-2}, that if $C_l \subseteq_i G,\ 4 \leq l \leq 6$, then computing a minimum dominating set is polynomial. When $G$ is $(claw, C_4,C_5,C_6,P_6)$-free then it is chordal. The Minimum Dominating Set problem is polynomial for $claw$-free chordal graphs.
\end{proof}

\begin{lemma}
Let $G$ be a connected $(claw,C_5,C_6,C_7,P_7)$-free graph. Computing a minimum dominating set is polynomial-time solvable.
\end{lemma}

\begin{proof}
By Properties \ref{contract} and \ref{leaf}, we can assume that all contractibles vertices of $G$ are contracted and that $G$ has no leaves.
By Lemma \ref{clawp6} we can assume that $P_6\subseteq_i G$. Let $P=v_1-v_2-v_3-v_4-v_5-v_6$.

Let $W=V\setminus N[V(P)]$. It follows from Property \ref{V=N[T]} that if $W=\emptyset$ then computing a minimum dominating set is polynomial. From now on $W\ne\emptyset$. Let $S=\{v \in V\setminus V(P)$ such that $2\le\vert N_P(v) \vert \leq 4\}$, $S_i \subseteq S$ being the set of vertices $v$ such that $\vert N_P(v)\vert=i$.
Let $H_i=\{v \in S_2: N_P(v) = \{v_i,v_{i+1}\},1\le i\le 5\}$. Since $G$ is claw-free each $H_i$ is complete. If there is an edge $r_ir_{i+1}$ with $r_i\in H_i,r_{i+1}\in H_{i+1}$ then $P=v_1-\cdots-v_i-r_i-r_{i+1}-v_{i+2}-\cdots-v_6=P_7$, a contradiction.
If there is an edge $r_ir_{j}$ with $r_i\in H_i,r_{j}\in H_{j}$ and $j\ge i+3$ then $C_{p}\subseteq_i G$, $p \geq 5$. So $H_1$ is anticomplete to $H_2,H_4,H_5$, the component $H_2$ is anticomplete to $H_3,H_5$, and the component $H_3$ is anticomplete to $H_4$. \\

We define $R_i$ as the set of vertices of $H_i$ having a neighbor in $W$, $R_i=\{v\in H_i: N_W(v) \ne\emptyset\},1\le i\le 5$. Since $G$ is $P_7$-free $R_1=R_5=\emptyset$. \\

Let $r\in R_i$, $r'\in R_i$, $r\ne r'$, $i\in\{2,4\}$ be such that $r$, respectively $r'$, has a neighbor $w\in W$, respectively $w'\in W$.
We show that $N_S(r)=N_S(r')$.

By contradiction we assume that there exists $s\in S$ such that $rs\in E$, $r's\not\in E$. From above $s\not\in R_i\cup H_{i-1}\cup H_{i+1}$. Let $i=2$ (the case $i=4$ is symmetric). Recall that $H_2$ is anticomplete to $H_1,H_3,H_5$, thus $s\in H_4\cup S_3\cup S_4$. If $s\in H_4$ then $G[\{r,w,v_3,s\}]$ is a claw, a contradiction. Hence $s\in S_3\cup S_4$. When $N_P(s)=\{v_1,v_2,v_3\}$ then $G[\{r',v_3,v_4,s\}]$ is a claw, a contradiction. When $N_P(s)=\{v_2,v_3,v_4\}$ then $G[\{r',v_1,v_2,s\}]$ is a claw, a contradiction. When $N_P(s)=\{v_3,v_4,v_5\}$ or $N_P(s)=\{v_4,v_5,v_6\}$ then $G[\{r,v_2,w,s\}]$ is a claw, a contradiction. So $s\in S_4$. When $N_P(s)=\{v_1,v_2,v_3,v_4\}$ then $G[\{r,v_1,v_4,s\}]$ is a claw, a contradiction. When $N_P(s)=\{v_2,v_3,v_4,v_5\}$ then $G[\{r,'v_1,v_2,s\}]$ is a claw, a contradiction. When $N_P(s)=\{v_3,v_4,v_5,v_6\}$ then $G[\{r,v_4,v_6,s\}]$ is a claw, a contradiction. Now let $i=3$. Recall that $H_3$ is anticomplete to $H_2,H_4$, thus $s\in H_1\cup H_5\cup S_3\cup S_4$. If $s\in H_1$ (the case $s\in H_5$ is symmetric) then $G[\{r,w,v_3,s\}]$ is a claw, a contradiction.  Hence $s\in S_3\cup S_4$. If $N_P(s)=\{v_1,v_2,v_3\}$ (the case $N_P(s)=\{v_4,v_5,v_6\}$ is symmetric) then $G[\{r,w,v_4,s\}]$ is a claw, a contradiction. If $N_P(s)=\{v_2,v_3,v_4\}$ (the case $N_P(s)=\{v_3,v_4,v_5\}$ is symmetric) then $G[\{r',v_4,v_5,s\}]$ is a claw, a contradiction. So $s\in S_4$.  When $N_P(s)=\{v_1,v_2,v_4,v_5\}$ or $N_P(s)=\{v_1,v_2,v_5,v_6\}$ then $G[\{r,v_1,v_5,s\}]$ is a claw, a contradiction. When $N_P(s)=\{v_2,v_3,v_5,v_6\}$ then $G[\{r',v_3,v_4,s\}]$ is a claw, a contradiction.
When $N_P(s)=\{v_1,v_2,v_3,v_4\}$ (the case $N_P(s)=\{v_3,v_4,v_5,v_6\}$ is symmetric) then $G[\{r',v_4,v_5,s\}]$ is a claw, a contradiction. Hence $N_P(s)=\{v_2,v_3,v_4,v_5\}$ but $G[\{r,v_2,v_5,s\}]$ is a claw, a contradiction. Thus $N_S(r)=N_S(r')$.\\

Let $r_2\in R_2$, $r_2'\in R_2$, $r_2\ne r_2'$ be such that $r_2$, respectively $r_2'$, has a neighbor $w\in W$, respectively $w'\in W$.
Let $r_4\in R_4$, $r_4'\in R_4$, $r_4\ne r_4'$ be such that $r_4$, respectively $r_4'$, has $w$, respectively $w'$, as neighbor.
We show that $N_{S\setminus H_4}(r_2)=N_{S\setminus H_4}(r'_2)$, respectively $N_{S\setminus H_2}(r_4)=N_{S\setminus H_2}(r'_4)$.

By contradiction we assume that there exists $s\in S$ such that $r_2s\in E,r'_2s\not\in E$. From above $s\not\in H_1\cup H_2\cup H_3$. When $s\in H_4$ we know that $s$ is not a neighbor of $w$. If $s\in H_4\cup H_5$ then $G[\{r_2,v_2,w,s\}]$ is a claw, a contradiction.  Hence $s\in S_3\cup S_4$. When $N_P(s)=\{v_1,v_2,v_3\}$ then $G[\{r_2',v_3,v_4,s\}]$ is a claw, a contradiction. When $N_P(s)=\{v_2,v_3,v_4\}$ then $G[\{r_2',v_1,v_2,s\}]$ is a claw, a contradiction. When $N_P(s)=\{v_3,v_4,v_5\}$ then $G[\{r_2,v_2,w,s\}]$ is a claw, a contradiction. When $N_P(s)=\{v_4,v_5,v_6\}$ then $G[\{r_2,v_4,v_6,s\}]$ is a claw, a contradiction. So $s\in S_4$. When $N_P(s)=\{v_1,v_2,v_4,v_5\}$ or $N_P(s)=\{v_1,v_2,v_5,v_6\}$ then $G[\{r,v_1,v_5,s\}]$ is a claw, a contradiction. When $N_P(s)=\{v_2,v_3,v_5,v_6\}$ then $G[\{r',v_3,v_4,s\}]$ is a claw, a contradiction. When $N_P(s)=\{v_1,v_2,v_3,v_4\}$ then $G[\{r_2,v_1,v_4,s\}]$ is a claw, a contradiction. When $N_P(s)=\{v_2,v_3,v_4,v_5\}$ then $G[\{r'_2,v_1,v_2,s\}]$ is a claw, a contradiction. When $N_P(s)=\{v_3,v_4,v_5,v_6\}$ then $G[\{r_2,v_4,v_6,s\}]$ is a claw, a contradiction. Thus $N_{S\setminus H_4}(r_2)=N_{S\setminus H_4}(r'_2)$ and by symmetry, for $r'_4\in R_4,r_4'\ne r_4,$ we have  $N_{S\setminus H_2}(r_4)=N_{S\setminus H_2}(r'_4)$.\\

Let $w\in W$. We show that $w$ cannot have two neighbors $r_i,r_{i+1}$ with $r_i\in R_i$, $r_{i+1}\in R_{i+1}$. Suppose for contradiction that these two neighbors exist. Then $v_1-\cdots-v_i-r_i-w-r_{i+1}-v_{i+2}-\cdots-v_6=P_8$, a contradiction. Now, since $R_1=R_5=\emptyset$, if $w$ has two neighbors $r_i\in R_i$, $r_j\in R_j$, $i\ne j,$ these two neighbors are $r_2 \in R_2$, $r_4 \in R_4$ and $r_2r_4\in E$, else $w-r_4-v_4-v_3-r_2-w=C_5$. Moreover, when $w$ has two neighbors $r_2\in R_2$, $r_4\in R_4,$ then for each neighbor $w'\in N_W(w)$, $w'$ has $r_2$ and $r_4$ as neighbors. Assume for contradiction that $w$ has a neighbor $w'\in W$ such that $w'r_2\not \in E$ (by symmetry $w'r_4\not \in E$ is the same case). Then $w'-w-r_2-v_3-\cdots-v_6=P_7$, a contradiction. It follows that $N[w]=N[w']$, a contradiction.

Hence setting $Z_{24}=\{w\in W: w \textrm{ has two neighbors } r_2\in R_2, r_4\in R_4\}$, $Z_{24}$ is an independent set.\\

Let $w,w'\in Z_{24}, w\ne w'$. Since $G$ is claw-free we have $N(w) \cap N(w')=\emptyset$.
We show that $N_{R_2}(w)$ is anticomplete to $N_{R_4}(w')$ and $N_{R_4}(w)$ is anticomplete to $N_{R_2}(w')$. By contradiction if $w$ has a neighbor $r_2\in R_2$, $w'$ has a neighbor $r_4\in R_4$, and $r_2r_4\in E$ then $G[\{v_2,r_2,w,r_4\}]$ is a claw, a contradiction. \\

Let $Z_i=\{w\in W: w \textrm{ has a neighbor in } R_i\setminus(N_{R_i}(Z_{24})\},\ 2\le i\le 4\}$.

We show that $Z_2,Z_3,Z_4$ are pairwise anticomplete. If there is an edge $w_2w_4,w_2\in Z_2$, $w_4\in Z_4,$ with $r_2'\in R_2$, $r_4'\in R_4$ the neighbors of $w_2,w_4$ respectively, then $w_2-r_2'-v_3-v_4-r_4'-w_4-w_2=C_6$ ($r_2'r_4'\not \in E$ else $G[\{v_2,r_2',w_2,r_4'\}]$ is a claw). If there is an edge $w_2w_3$, $w_2\in Z_2$, $w_3\in Z_3,$ with $r_2'\in R_2$, $r_3'\in R_3$ the neighbors of $w_2,w_3$ respectively, then $w_2-r_2'-v_3-r_3'-w_3-w_2=C_5$ (recall $r_2'r_3' \not\in E$). By symmetry there is no edge between $Z_3$ and $Z_4$.\\

Let $Y=W\setminus(Z_2\cup Z_3\cup Z_4\cup Z_{24})$. One can observe that for every $w\in Y$ we have $N_{Z_2}(w) = N_{Z_4}(w) = N_{Z_{24}}(w)= \emptyset$ else $P_7\subseteq_i G$.

Let $Y_3=\{w\in Y: w \textrm{ has a neighbor in } Z_3\}$. If there exists $w'\in Y\setminus Y_3$ such that $w'$ has a neighbor $w,w\in Y_3$, then $P_7\subseteq_i G$. Hence $Y=Y_3$. \\

We show that we can assume that $Z_2,Z_4,Y_3$ are three independent sets. The arguments are the same for the three sets, so we show that $Z_2$ is an independent set. For contradiction, we assume that there are $w_1,w_2\in Z_2$ such that $w_1w_2\in E$.
We prove that $N_{R_2}(w_1) = N_{R_2}(w_2)$. If $N_{R_2}(w_1) \ne N_{R_2}(w_2)$ then there exists $r_2\in R_2$ which is a neighbor of $w_1$ but not a neighbor of $w_2$. Then $w_2-w_1-r_2-v_3-\cdots-v_6=P_7$, a contradiction. If $N_{Z_2}(w_1) \ne N_{Z_2}(w_2)$ then there exists $w_3\in Z_2$ such that $w_2w_3\in E$, $w_1w_3\not\in E,$ but $G[\{v_2,r_2,w_1,w_3\}]$ is a claw, a contradiction. Hence $N[w_1]=N[w_2]$, a contradiction. Hence $Z_2,Z_4,Y_3$ are three independent sets. \\

Since $G$ is claw-free then for every two distinct vertices $w_1,w_2 \in Z_2\cup Z_4\cup Y_3$ we have $N(w_1) \cap N(w_2)=\emptyset$.

We prove that for every $w\in Y_3$, $N(w)$ is a clique. Let $w\in Z_3$. Suppose there are $s,s'$ two non adjacent vertices in $N(w)$. Since $G$ is claw-free $s,s'$ cannot have a common neighbor in $R_3$. Let $r\in R_3$ be a neighbor of $s$. Then $s'-w-s-r-v_3-v_2-v_1=P_7$, a contradiction. \\

Since $G$ is claw-free, if there are a vertex $r \in R_i$ with a neighbor $z \in Z_i$ and a vertex $s \in S$ such as $sz \not\in E$ and $v_i \not\in N(s)$ then $G$ contains a claw, a contradiction, (note that $v_{i+1} \not\in N(s)$ is symmetric). Hence $N(Z_i)$ is anticomplete to $H_j$, $j \neq i$. \\

We show that we can assume that $Z_2=Z_4=\emptyset$. The arguments are the same in the two cases, so we consider $Z_2$.
Let $r,r'\in R_2$ be two neighbors of $w\in Z_2$. We show that $N[r]=N[r']$. Since $N_R(w)=N_{R_2}(w)$ and $rr'\in E$  then, as proved above, $N_S(r)= N_S(r')$. For two distinct $w_1,w_2 \in Z_2$, $N(w_1) \cap N(w_2)=\emptyset$. Hence, $N[r]=N[r']$, a contradiction. Then $w$ is a leaf, a contradiction. \\

Now we study the structure of $Z_3$. For every distinct two vertices $w_1,w_2\in Z_3$ such that $w_1w_2\in E$, there cannot exist two distinct vertices $w_1',w_2'\in Z_3$ such that $w_1w_1'\in E$, $w_1'w_2\not\in E$ and $w_2w_2'\in E$, $w_1w_2'\not\in E$. For contradiction we suppose that such two vertices exist. We assume first that $w_2$ has a neighbor $r_2\in R_3$ such that $r_2w_1\not\in E$. If $w_1'r_2\not\in E$ then $v_1-v_2-v_3-r_2-w_2-w_1-w_1'=P_7$ else $G[\{v_4,r_2,w_2,w_1'\}]$ is a claw, a contradiction. So $w_1,w_2$ have a common neighbor $r_1\in R_3$. If $w_1'r_1\in E$ then $G[\{v_3,r_1,w_2,w_1'\}]$ is a claw, a contradiction. Thus $w_1'r_1\not\in E$ and $w_1'$ has a neighbor $r_1'\in R_3,r_1'\ne r_1$. If $r_1'w_2\in E$ then $G[\{v_3,r_1',w_2,w_1'\}]$ is a claw, a contradiction. So $r_1'w_2\not\in E$. If $r_1'w_1\not\in E$ then $v_1-v_2-v_3-r_1'-w_1'-w_1-w_2=P_7$, a contradiction. Thus $r_1'w_1\in E$. If $r_1'w_2'\not\in E$ then $v_1-v_2-v_3-r_1'-w_1-w_2-w_2'=P_7$, a contradiction. So $r_1'w_2'\in E$ but $G[\{v_4,r_1',w_1,w_2'\}]$ is a claw, a contradiction.

As a consequence each connected component $A_i$ of $Z_3$ has a universal vertex. Also, $G$ being claw-free two distinct connected components cannot share a neighbor in $R_3$. Moreover, by Property \ref{leaf} we have assumed that each $w_3\in Z_3$ is not a leaf. \\

We show that $w\in Y_3$ is connected to a universal vertex of a connected component $A_i$ of $Z_3$. We assume that the neighbors of $w$ are not universal in $A_i$. Let $s\in A_i$ be a neighbor of $w$, let $u,u\ne s,$ be a universal vertex of $A_i$. Since $s$ is not universal there exists $v,v\in A_i$ such that $sv\not\in E$ and $uv\in E$. Since $N(w)$ is complete $wv\not\in E$.
Let $r\in R_3$ be a neighbor of $s$. Since $G$ is claw-free then $rv\not\in E$. Let $r',r'\in R_3,r'\ne r,$ be a neighbor of $v$. As just above $r's\not\in E$. If $r'u\not\in E$ then $v_1-v_2-v_3-r'-v-u-s=P_7$ else $v_1-v_2-v_3-r'-u-s-w=P_7$, a contradiction. \\

We are ready to show how to build a $\gamma$-set in polynomial time.

First, we treat the case where $Z_{24}\ne \emptyset$. Let $r_2\in R_2,\ r_4\in R_4$ be two neighbors of $w,\ w\in Z_{24}$.

We show that $R_3=\emptyset$. Assume that there exists $w'\in W$ with a neighbor $r_3\in R_3$. Since $w'$ is not a neighbor of $r_2$ or $r_4$ we have $w'-r_3-v_3-r_2-r_4-v_5-v_6=P_7$, a contradiction. So $R_3=\emptyset$ and since $Z_2=Z_4=\emptyset$ we have $W=Z_{24}$.

Recall that $W=Z_{24}$ is independent and that for every two distinct vertices $w', w'\in Z_{24}$ we have $N(w)\cap N(w')=\emptyset$.\\

The $\gamma$-set is build as follows:

By Property \ref{V=N[T]+W}, we can assume that $\vert W \vert\ge 2$. We take $r_2 \in R_2$ a neighbor of $w$ (recall that the neighbors of $w$ in $R_i,i\in\{2,4\},$ have the same neighborhood and that all vertices of $R_i$ have the same neighbors in $S\setminus H_4$), and for each other $w'\in Z_{24}$ we take one adjacent vertex $r'_4\in R_4$. These vertices dominate $Z_{24}\cup H_2\cup H_4\cup\{v_2,v_3,v_4,v_5\}$.
At least one more vertex is necessary to dominate $G$ since $v_1$ and $v_6$ are not dominated. Adding the three vertices $v_2,v_4,v_6$ we have a dominating set (not necessarily minimum). We check first if there exists $s$ a neighbor of both $v_1$ and $v_6$ that dominates the rest of the graph. If such vertex $s$ does not exist, checking for all the pairs $s_1,s_6$ where $s_i$ is a neighbor of $v_i,i\in \{1,6\}$, one can verify if there is a $\gamma$-set with only two more vertices (note that there are at most $O(n^2)$ of such pairs).\\

Now we deal with the case $Z_{24}=\emptyset$.

The $\gamma$-set is build as follows:
\begin{itemize}
\item $Y_3\ne\emptyset$. For each $w\in Y_3$ we take one universal vertex in the connected component $A_i$ of $Z_3$ connected to $w$. For each connected component $A_i$ of $Z_3$ that is not connected to a vertex of $Y_3$, we do as follows: if there exists $r_3\in R_3$ which is complete to $A_i$ (recall that such vertices have the same neighborhood)  then we take $r_3$, else we take one universal vertex of $A_i$. These vertices dominate $Y_3\cup Z_3$.
At least one more vertex is necessary to dominate $G$ since $v_1$ and $v_6$ are not dominated. Adding the three vertices $v_2,v_4,v_6$ we have a dominating set (not necessarily minimum). We check first if there exists $s$ a neighbor of both $v_1$ and $v_6$ that dominates the rest of the graph. If such vertex $s$ is not found, checking for all the pairs $s_1,s_6$ where $s_i$ is a neighbor of $v_i,i\in \{1,6\}$, one can verify if there is a $\gamma$-set with only two more vertices (note that there are at most $O(n^2)$ such pairs).

\item $Y_3 = \emptyset$. Thus $Z_3=W$. For every connected component $A_i$ of $Z_3$, if there exists $r_3\in R_3$ which is complete to $A_i$ (recall that such vertices have the same neighborhood) then we take $r_3$, else we take one universal vertex of $A_i$. These vertices dominate $Z_3$.
At least one more vertex is necessary to dominate $G$ since $v_1$ and $v_6$ are not dominated. Adding the three vertices $v_2,v_4,v_6$ we have a dominating set (not necessarily minimum). We check first if there exists $s$ a neighbor of both $v_1$ and $v_6$ that dominates the rest of the graph. If such vertex $s$ is not found, checking for all the pairs $s_1,s_6$ where $s_i$ is a neighbor of $v_i$, $i\in \{1,6\}$, one can verify if there is a $\gamma$-set with only two more vertices (note that there are at most $O(n^2)$ such pairs).
\end{itemize}
Clearly the construction of the $\gamma$-set is polynomial.
\end{proof}

\begin{coro}\label{clawp7}
The Minimum Dominating Set problem is polynomial for $(claw,P_7)$-free graphs.
\end{coro}

\begin{lemma}\label{clawp8}
Let $G$ be a connected $(claw,C_6,C_7,C_8,P_8)$-free graph. If $C_5\subseteq_i G$ then computing a minimum dominating set is polynomial.
\end{lemma}

\begin{proof}
By Properties \ref{contract} and \ref{leaf}, we can assume that all contractibles vertices of $G$ are contracted and that $G$ has no leaves.
Let $C=v_1-v_2-v_3-v_4-v_5-v_1=C_5\subseteq_i G$. Let $W=V\setminus N[V(C)]$. It follows from  Property \ref{V=N[T]} that if $W=\emptyset$ then computing a minimum dominating set is polynomial. From now on $W\ne\emptyset$.

Let $S=\{v \in V\setminus V(C)$ : $2\le\vert N(v)\cap V(C)\vert \leq 5\}$, and $S_i \subseteq S$ being the set of vertices $v$ such that $\vert N(v)\cap V(C)\vert=i$. Let $H_i=\{v \in S_2: N(v) \cap V(C) = \{v_i,v_{i+1}\}\}$, $1\le i\le 5$ (for convenience $v_{5+1}$ stands for $v_1$). Since $G$ is claw-free, each $H_i$ is complete. Moreover, if there is an edge $r_ir_{i+1}$ with $r_i\in H_i$, $r_{i+1}\in H_{i+1}$ then $r_i-v_i-v_{i-1}-\cdots-v_{i+2}-r_{i+1}-r_i=C_6$, a contradiction. Hence $H_i$ is anticomplete to $H_{i+1}$. We define $R_i$ as the set of vertices of $H_i$ having a neighbor in $W$, $R_i=\{v\in H_i: N_W(v) \neq \emptyset\},\ 1\le i\le 5$, $R=R_1\cup\cdots\cup R_5$.\\

Since $W\neq \emptyset$, we assume that there exists $w_1\in W$ such that $w_1$ has a neighbor $r_1\in R_1$. Suppose that $R_2\ne \emptyset$. There exists $w_2\in W$ with a neighbor $r_2\in R_2$. If $w_1=w_2$ then $w_1-r_2-v_3-v_4-v_5-v_1-r_1-w_1=C_7$, a contradiction. So $w_1\ne w_2$. If $w_1w_2\in E$ then $w_1-w_2-r_2-v_3-v_4-v_5-v_1-r_1-w_1=C_8$ else $w_1-r_1-v_1-v_5-v_4-v_3-r_2-w_2-w_1=P_8$, a contradiction. So, if $R_i \neq \emptyset$ then $R_{i-1}=R_{i+1}=\emptyset$. Hence $R_2=R_5=\emptyset$. \\

Let $r\in R_i$, $r'\in R_i$, $r\ne r'$, $i\in\{1,3,4\}$ be such that $r$, respectively $r'$, has a neighbor $w\in W$, respectively $w'\in W$.
We show that $N_S(r)=N_S(r')$.

By contradiction we assume that there exists $s\in S$ such that $rs\in E$, $r's\not\in E$. From above $s\not\in H_i\cup H_{i-1}\cup H_{i+1}$.
Let $i=1$. If $s\in H_3\cup H_4$ then $G[\{r,v_1,w,s\}]$ is a claw, a contradiction. Thus $s\in S_3\cup S_4\cup S_5$. When $N_C(s)=\{v_1,v_2,v_3\}$ (the case $N_C(s)=\{v_1,v_2,v_5\}$ is symmetric) then $G[\{r',v_5,w,s\}]$ is a claw, a contradiction. When $N_C(s)=\{v_2,v_3,v_4\}$ (the case $N_C(s)=\{v_1,v_4,v_5\}$ is symmetric) then $G[\{r,v_1,w,s\}]$ is a claw, a contradiction. When $N_C(s)=\{v_3,v_4,v_5\}$ then $G[\{r,v_3,v_5,s\}]$ is a claw, a contradiction. So $s\in S_4\cup S_5$. When $N_C(s)=\{v_1,v_2,v_3,v_4\}$ (the case $N_C(s)=\{v_1,v_2,v_4,v_5\}$ is symmetric) then $G[\{r',v_1,v_5,s\}]$ is a claw, a contradiction. When $N_C(s)=\{v_1,v_3,v_4,v_5\}$ or $N_C(s)=\{v_2,v_3,v_4,v_5\}$ or $N_C(s)=\{v_1,v_2,v_3,v_5\}$ or $s\in S_5$ then $G[\{r,v_3,v_5,s\}]$ is a claw, a contradiction.  For $i=3$ and $i=4$ the arguments are the same. Thus $N_S(r)=N_S(r')$.\\

Let $r_1\in R_1$, $r_1'\in R_1$, $r_1\ne r_1'$ be such that $r_1$, respectively $r_1'$, has a neighbor $w\in W$, respectively $w'\in W$.
Let $r_3\in R_3$, $r_3'\in R_3$, $r_3\ne r_3'$ be such that $r_3$, respectively $r_3'$, has $w$, respectively $w'$, as neighbor.
We show that $N_{S\setminus H_3}(r_1)=N_{S\setminus H_3}(r'_1)$, respectively $N_{S\setminus H_1}(r_3)=N_{S\setminus H_1}(r'_3)$.

Let $i=1$. By contradiction we assume that there exists $s\in S\setminus R_3$ such that $r_1s\in E$, $r'_1s\not\in E$. From above $s\not\in H_1\cup H_{2}\cup H_{5}$. If $s\in H_4$ then $G[\{r_1,v_1,w,s\}]$ is a claw, a contradiction. So $s\in S_3\cup S_4\cup S_5$. When $N_C(s)=\{v_1,v_2,v_3\}$ (the case $N_C(s)=\{v_1,v_2,v_5\}$ is symmetric) then $G[\{r_1',v_1,v_5,s\}]$ is a claw, a contradiction. When $N_C(s)=\{v_2,v_3,v_4\}$ (the case $N_C(s)=\{v_1,v_4,v_5\}$ is symmetric) then $G[\{r_1,v_1,w,s\}]$ is a claw, a contradiction. When $N_C(s)=\{v_3,v_4,v_5\}$ then $G[\{r_1,v_3,v_5,s\}]$ is a claw, a contradiction. So $s\in S_4\cup S_5$. When $N_C(s)=\{v_1,v_2,v_3,v_4\}$ (the case $N_C(s)=\{v_1,v_2,v_4,v_5\}$ is symmetric) then $G[\{r'_1,v_1,v_5,s\}]$ is a claw, a contradiction. When $N_C(s)=\{v_1,v_3,v_4,v_5\}$ or $N_C(s)=\{v_2,v_3,v_4,v_5\}$ or $N_C(s)=\{v_1,v_2,v_3,v_5\}$ or $s\in S_5$ then $G[\{r_1,v_3,v_5,s\}]$ is a claw, a contradiction. By symmetry the arguments are the same for $i=3$. Hence $N_{S\setminus H_3}(r_1)=N_{S\setminus H_3}(r'_1)$ and $N_{S\setminus H_1}(r_3)=N_{S\setminus H_1}(r'_3)$. \\

We study the case where $w_1$ has a neighbor $r_i$, $r_i \in R_i$, $i \in \{3, 4\}$. Since both cases are symmetric, let $r_3$, $r_3 \in R_3,$ be a neighbor of $w_1$. If $r_1r_3 \not\in E$ then $w_1-r_1-v_1-v_5-v_4-r_3-w_1=C_6$, a contradiction. Hence $N_{R_1}(w_1)$ is complete to $N_{R_3}(w_1)$. Since $R_3 \neq \emptyset$, we have $R_4 = \emptyset$ and $N_R(w_1)\subseteq R_1 \cup R_3$. Hence, we define the following subsets of $W$:
\begin{itemize}
\item $Z = \{w \in W : N_R(w) \neq \emptyset\}$;
\item $Z_i = \{z \in Z : N_{R_i}(z) \neq \emptyset, N_{R_j}(z) = \emptyset, 1 \leq i \leq 5, i \neq j\}$;
\item $Z_{ij} = \{z \in Z : N_{R_i}(z) \neq \emptyset, N_{R_j}(z) \neq \emptyset, 1 \leq i<j \leq 5\}$;
\item $Y=W\setminus Z$.
\end{itemize}

First, we show that $Z_i$ is anticomplete to $Z_{ij}$, then we show that $Z_i$ consists of leaves (so is empty). We conclude that $Z_{ij} \neq \emptyset$ implies $Z=Z_{ij}$. We set $w_1\in Z_{13}$, and since all cases are symmetric, we focus on $Z_1 \neq \emptyset$. \\

Let $w_1' \in Z_1$ with a neighbor $r_1'$, $r_1' \in R_1$, $r_1' \neq r_1$. Note that $r_1'r_3 \not\in E$ else $G$ contains a claw. If $w_1w_1' \in E$, then $r_1'w_1 \in E$, else $w_1-w_1'-r_1'-v_1-v_5-v_4-r_3-w_1=C_7$, but $w_1-r_1'-v_1-v_5-v_4-r_3-w_1=C_6$, a contradiction. Hence, $w_1w_1'\not\in E$ (by symmetry, for every $w_3 \in Z_3$, $w_1w_3 \not\in E$). Thus $Z_1$ and $Z_3$ are anticomplete to $Z_{13}$. \\

Now, we show that the vertices of $Z_1$ are leaves. Assume that there exists $v \in N(w_1')$, $v\ne r_1'$ such that $N[v] \neq N[w_1']$. If $v \in Z_3$ then $v-w_1'-r_1'-v_1-v_5-v_4-r_3-w_1=P_8$, a contradiction. If $v \in Y$ then either $v-w_1'-r_1'-v_1-v_5-v_4-r_3-w_1=P_8$ or $v-w_1'-r_1'-v_1-v_5-v_4-r_3-w_1-v=C_8$, a contradiction. If $v \in Z_1$ and $r_1'v \not\in E$ then $v-w_1'-r_1'-v_1-v_5-v_4-r_3-w_1=P_8$, a contradiction. Hence $N_{R_1}(w_1')=N_{R_1}(v)$. Since $N[v] \neq N[w_1']$ we can assume that there exists $v' \in W$ such that $vv' \in E$ but $v'w_1' \not\in E$. Yet with the same arguments as before we have $N_{R_1}(v)=N_{R_1}(v')$ and since $w_1'v' \not\in E$ then $G[\{r_1', v', v_1, w_1'\}]$ is a claw, a contradiction. Thus $Z_1$ consists of leaves, a contradiction. Thus $Z_1 = \emptyset$ and by symmetry $Z_3 = \emptyset$. So $Z=Z_{13}$. \\

We show that every pair $v,v' \in Z_{13}$ with $vv'\in E$ satisfy $N_{R_1 \cup R_3}(v)=N_{R_1 \cup R_3}(v')$. Let $w_1' \in Z_{13}$ be a neighbor of $w_1$. Suppose that there exists $r_1' \in N_{R_1}(w_1')$ such that $r_1'w_1 \not\in E$. If $r_1'r_3 \in E$ then $G[\{r_1', r_3, v_3, w_1\}]$ is a claw, a contradiction. If $r_3w_1' \not\in E$ then $w_1-r_3-v_4-v_5-v_1-r_1'-w_1'-w_1=C_7$, else $w_1'-r_1'-v_1-v_5-v_4-r_3-w_1'=C_6$, a contradiction. Hence, $N_{R_1}(w_1)=N_{R_1}(w_1')$ and by symmetry $N_{R_3}(w_1)=N_{R_3}(w_1')$. \\

Suppose that $Y \neq \emptyset$. Let $y \in Y$ be a neighbor of $w_1$. We show that $Z_{13}$ is a clique. Let $w_1' \in Z_{13}$ such that $w_1w_1' \not\in E$. We have $N_{R_1 \cup R_3}(w_1) \cap N_{R_1 \cup R_3}(w_1') = \emptyset$ else $G$ contains a claw. Yet, there exists $r_1'$ a neighbor of $w_1'$ in $R_1$ such that either $y-w_1-r_3-v_4-v_5-v_1-r_1'-w_1'=P_8$ or $y-w_1-r_3-v_4-v_5-v_1-r_1'-w_1'-y=C_8$, a contradiction, (note that $r_1'r_3 \not\in E$ else $G$ contains a claw). Hence $Z_{13}$ is a clique. \\

We show that the vertices of $Y$ are leaves. Suppose that $y$ has a neighbor $y' \in Y$. If $y'w_1\not\in E$ then $y'-y-w_1-r_1-v_1-v_5-v_4-v_3=P_8$, a contradiction. Hence $N_Z(y)=N_Z(y')$. Since we assume that $N[y] \neq N[y']$, there exists $v,v\in Y,$ such that $vy\in E$, $vy'\not\in E$. From above $w_1v\in E$ but $G[\{r_1,w_1,y',v\}]$ is a claw, a contradiction.
Thus, $Y$ is an independent set. Now, $N(y)\subseteq Z_{13}$ is a clique. Since for every two vertices $w_1$, $w_1' \in Z_{13}$ we have $N_{R_1 \cup R_3}(w_1)=N_{R_1 \cup R_3}(w_1')$ we can assume that $N(y)$ can be contracted into an unique vertex. Thus, $Y$consists of leaves, a contradiction. Hence $Y=\emptyset$. \\

As shown before, every two neighbors of $Z_{13}$ have the same neighbors in $R$, so they can be contracted and we can assume that $Z_{13}$ is an independent set. Moreover, since $G$ is $claw$-free, for every two distinct $z, z' \in Z_{13}$, $N[z] \cap N[z'] = \emptyset$. Also, recall that the neighbors of each $z, z \in Z_{13}$ induce a clique.\\

We show how to build a $\gamma$-set of $G$. Recall that $W=Z_{13}$. By Property \ref{V=N[T]+W} we can assume that $\vert W\vert\ge 2$. So there are $w_1$, $w_1'\in Z_{13}$ with neighbors $r_1$, $r_1'\in R_1$ and $r_3$, $r_3'\in R_3$, respectively. Let $q = \vert Z_{13} \vert$. Clearly, to dominate $Z_{13}$ we must take $q$ vertices. We take $r_1$ and $r_3'$. Recall that the vertices of $R_1$ and $R_3$ have the same neighbors in $S \cup V(C)$. Then, we take the $q-2$ vertices of $w\in Z_{13}$, $w\ne w_1,w_1'$. These $q$ vertices dominate $\{v_1, v_2, v_3, v_4\} \cup H_1 \cup H_3 \cup Z_{13}$.
It remains to dominate some vertices of $H_2 \cup H_4 \cup H_5 \cup S_3 \cup S_4 \cup \{v_5\}$. If there exists a vertex $v,v\in S \cup \{v_5\},$ which is universal to these non dominated vertices we take $v$, else we take the vertices $\{v_2, v_5\}$. \\

Now, we assume that $Z_{ij} = \emptyset$. Hence let $w_1 \in Z_1$. We study the case $R_3\ne \emptyset$. Recall that $R_2=R_4=R_5=\emptyset$. Let $w_3\in W$ such that $w_3$ has a neighbor $r_3\in R_3$. If $w_1w_3\in E$ then $w_1-w_3-r_3-v_4-v_5-v_1-r_1-w_1=C_7$ ($r_1r_3\not\in E$ else $G$ contains a claw), a contradiction. So $Z_1$ is anticomplete to $Z_3$. We assume that $w_1$ has a neighbor $v \in Y$. If $vw_3\in E$ then $v-w_3-r_3-v_4-v_5-v_1-r_1-w_1-v=C_8$ else $v-w_1-r_1-v_1-v_5-v_4-r_3-w_3=P_8$, a contradiction. Hence every neighbor $w_1'$, $w_1'\in W,$ of $w_1$ is in $Z_1$.  If $w_1'r_1\not\in E$ then $w_1'-w_1-r_1-v_1-v_5-v_4-r_3-w_3=P_8$, a contradiction. Hence $N_R(w_1) = N_R(w_1')$. Since $G$ is claw-free, for every $r \in R$, $N_W(r)$ is a clique, thus $N[w_1]=N[w_1']$, a contradiction. So $Z_1$ is an independent set. Now, recall that for every pair of vertices $r, r' \in R_i$, $1 \leq i \leq 5$, $N_S(r)=N_S(r')$. Hence, when $r, r'\in R_1$ have a common neighbor in $Z_1$, we have $N[r]=N[r']$, a contradiction. Hence $Z_1$ consists of leaves, a contradiction. Also, by symmetry, $W=Z_1\cup Z_3=\emptyset$, a contradiction. \\

Now we focus on $R_3=R_4=\emptyset$ (note that $Z=Z_1$).

We study the case where $H_2\ne\emptyset$ or $H_5\ne\emptyset$. Let $v\in H_2$ (the case $v\in H_5$ is symmetric). We have $Y = \emptyset$, else there are $y \in Y$, $z \in Z_1$, $r \in R_1$ such that $v-v_3-v_4-v_5-v_1-r-z-y=P_8$ (recall that $vr \not\in E$). So $W=Z=Z_1$. Let $w_1,w_2\in W$. We assume that $w_1w_2\in E$. Recall that $N[w_1]\ne N[w_2]$. Let $w_1r_1$, $w_2r_2\in E$, $r_1\ne r_2,$ such that $w_1r_2\not\in E$. We have $v-v_3-v_4-v_5-v_1-r_2-w_2-w_1=P_8$, a contradiction. Recall that for every $r \in R$, $N_W(r)$ is a clique, thus $N[w_1] = N[w_2]$, a contradiction. Hence $W$ is an independent set. Moreover, for every $w \in W$ and $r,r' \in N(w)$ we know that $r$ and $r'$ share the same neighbors in $V(C) \cup S$. Hence $W$ is composed exclusively of leaves, so $W = \emptyset$, a contradiction. \\

Now we can assume that $H_2=H_5=\emptyset$. Let $Z_A \subset Z$, $Z_A=\{w\in W : N_Y(w)=\emptyset\}$. We show that each connected component $A_i$ of $G[Z_A]$ contains a universal vertex relatively to $A_i$. For contradiction we suppose that there exists $A_i$, $A_i \subseteq Z_A$ with no universal vertex in it. Assume that $z_1-z_2-z_3-z_4=P_4\subseteq_i A_i$. Let $r, r\in R_1,$ be a neighbor of $z_1$ (note that there is a $P_5$ from $v_3$ to $r$).

Since $G$ is claw-free $rz_3,rz_4\not\in E$. If $rz_2\in E$ then there is a $P_8$ from $v_3$ to $z_4$ else there is a $P_8$ from $v_3$ to $z_3$, a contradiction. Now, we assume that $z_1-z_2-z_3-z_4-z_1=C_4\subseteq_i A_i$. Let $r$, $r\in R_1,$ be a neighbor of $z_1$. Since $G$ is claw-free we have $rz_3\not\in E$. If $rz_2 \in E$ then $rz_4 \not\in E$ else $G$ contains a claw, but $v_3-v_4-v_5-v_1-r-z_2-z_3-z_4=P_8$, a contradiction. If $rz_2 \not\in E$ then $v_3-v_4-v_5-v_1-r-z_1-z_2-z_3=P_8$, a contradiction. So $A_i$ is $(C_4,P_4)$-free. It follows that there are $z_1-z_2-z_3=P_3\subseteq_i A_i$ and $z_4\in A_i$ such that $z_4z_1,z_4z_2,z_4z_3\not \in E$. Also there exists $z\in A_i$ such that $z_2-z-z_4$ and $zz_1,zz_3\in E$ but $A_i[\{z,z_1,z_3,z_4\}]$ is a claw, a contradiction. So each $A_i$ has a universal vertex. Clearly, for two distinct components $A_i,A_j$ we have $N_{R_1}(A_i)\cap N_{R_1}(A_j)=\emptyset$ else there is a claw. \\

Suppose that $Y \neq \emptyset$. We show that $Y$ is an independent set. Suppose that there are $y,y'\in Y$ with $yy'\in E$. Recall that $N[y]\ne N[y']$. If $N_{Z_1}(y)\ne N_{Z_1}(y')$ then, w.l.o.g, $yw_1\in E$, $y'w_1\not\in E$, but $v_3-v_4-v_5-v_1-r_1-w_1-y-y'=P_8$, a contradiction. So $N_{Z_1}(y)= N_{Z_1}(y')$. There is no vertex $y'' \in Y$ such that $yy'' \in E$, $y'y'' \not\in E$, else $G$ contains a claw.
Hence $Y$ is an independent set and for every pair of vertices $y,y'\in Y$ we have $N(y)\cap N(y')=\emptyset$.

We show that for every $y\in Y$ its neighborhood $N(y)$ is a clique. For contradiction we assume that $y$ has two neighbors $z_1,z_2 \in Z$, $z_1z_2\not\in E$. Since $G$ is claw-free $z_1$ and $z_2$ cannot have a common neighbor in $R_1$. Let $r$, $r\in R_1,$ be a neighbor of $z_1$. Then $v_3-v_4-v_5-v_1-r-z_1-y-z_2=P_8$, a contradiction. Hence, $Y$ is an independent set, for each $y, y \in Y$, $N(y)$ is a clique. So we suppose $\vert N(y) \vert \geq 2$, else $y$ is a leaf.\\

We show that we can assume that each connected component $A_i$ of $G[Z_A]$ is anticomplete to $N(Y)$. Since $Y$ has no leaves, let $y \in Y$ with two neighbors $z,z' \in
Z_1$ such that $N[z] \neq N[z']$. Suppose that there exists $u \in Z_A$ a neighbor of $z$. First, we assume that $N_R(z) \neq N_R(z')$. W.l.o.g. let $r,r' \in R_1$ be respectively the neighbors of $z, z'$ such that $r'z, rz' \not\in E$. If $uz'\not\in E$ then $ur'\not\in E$ else $G$ contains a claw, but then $u-z-z'-r'-v_1-v_5-v_4-v_3=P_8$, a contradiction. Hence $uz', r'u \in E$ but $y-z-u-r'-v_1-v_5-v_4-v_3=P_8$, a contradiction. So $N_R(z)= N_R(z')$. Second, we assume that $N_Z[z] \neq N_Z[z']$. W.l.o.g. $uz' \not\in E$. Let $r \in R_1$ a neighbor of both $z, z'$. Clearly $ru\not\in E$ else $G$ contains a claw, but $G[\{r, u, y, z\}]$ is a claw, a contradiction. So we can assume that each $A_i$ is anticomplete to $N(Y)$. \\

We construct a $\gamma$-set as follows:

Let $q=\vert Y\vert$ and $k$ be the number of connected components of $Z_A$. Clearly, $q$ vertices are necessary to dominate $Y$. So for each $y_i\in Y$ we will take one of its neighbor as follows. Let us denote $R_1(y_i)=N_{R_1}(N(y_i))$. If $y_i$ has a neighbor $z_i$ which is complete to $R_1(y_i)$ then we take $z_i$, else we take every arbitrary neighbor of $y_i$ (recall that in both cases these $y_i$ have the same neighbors in $Z$). These $q$ vertices dominate $Y \cup (Z \setminus Z_A)$ and some of the vertices in $R_1(Y)$.

Now $k$ vertices are necessary to dominate $Z_A$. For each component $A_i\subset Z_A$ we do as follows. If there exists $r\in R_1$ which is complete to $A_i$ we take $r$ into the $\gamma$-set (case $a$), else we take one universal vertex of $A_i$ (case $b$) (recall that in both cases these $r$ have the same neighbors in $S$).

These $k$ vertices dominate $Z_A\cup H_1\cup\{v_1,v_2\}$ if at least one vertex is chosen in the case $a$, else they dominate $Z_A$. \\

Case where at least one vertex is chosen with the case $a$: $v_3,v_4,v_5$ are not dominated with the $q+k$ already chosen vertices ($H_1$ is complete thus $r\in R_1$ dominates $H_1 \cup \{v_1, v_2\}$). So a dominating set of $G$ has size at least $q+k+1$. Adding the two vertices $v_3$ and $v_5$, we have a dominating set (not necessarily minimum). Checking if there exists a vertex $v \in V(C) \cup S$, that is universal to the remaining non-dominated vertices, can be done in polynomial-time. \\

Case where all the vertices are chosen with the case $b$: it remains to dominate $C$ and some vertices of $S_2 \cup S_3 \cup S_4$. So a dominating set of $G$ has a size at least $q+k+1$. Adding the three vertices $v_1, v_3, v_5$, we have a dominating set (not necessarily minimum). If there exists a vertex $v \in S_5$ that is universal to the remaining non-dominated vertices we take it. If no such vertex exists, checking for all the pairs $\{v, v'\}\subset N[V(C)]$, one can verify if there exists a $\gamma$-set with $q+k+2$ vertices (note that there are at most $O(n^2)$ of such pairs).
\end{proof}

\begin{lemma}\label{P8C}
Let $G$ be a connected $(claw,C_5,C_6,C_7,C_8,P_8)$-free graph. Computing a minimum dominating set is polynomial-time solvable.
\end{lemma}

\begin{proof}
By Lemma \ref{clawp7} we can assume that $P_7\subseteq_i G$. Let $P=v_1-v_2-v_3-v_4-v_5-v_6-v_7$. By Properties \ref{contract} and \ref{leaf}, we can assume that all contractibles vertices of $G$ are contracted and that $G$ has no leaves.

Let $W=V\setminus N[V(P)]$. By Property \ref{V=N[T]} if $W=\emptyset$ then computing a minimum dominating set is polynomial. From now on $W\ne\emptyset$. Let $S=\{v \in V\setminus V(P)$ : $2\le\vert N(v)\cap P\vert \leq 4\}$, and $S_i \subseteq S$ being the set of vertices $v$ such that $\vert N(v)\cap V(P)\vert=i$. Let $H_i=\{v \in S_2: N(v) \cap V(P) = \{v_i,v_{i+1}\},1\le i\le 6\}$. Since $G$ is claw-free each $H_i$ is complete.
If there is an edge $r_ir_{i+1}$ with $r_i\in H_i$, $r_{i+1}\in H_{i+1}$ then $P=v_1-\cdots-v_i-r_i-r_{i+1}-v_{i+2}-\cdots-v_7=P_8$, a contradiction.
If there is an edge $r_ir_{j}$ with $r_i\in H_i$, $r_{j}\in H_{j}$ and $j\ge i+3$ then $C_{p}\subseteq_i G$, $p \geq 5$, a contradiction. So $H_1$ is anticomplete to $H_2,H_4,H_5,H_6$, and $H_2$ is anticomplete to $H_3,H_5,H_6$, and $H_3$ is anticomplete to $H_4,H_6$. \\

We define $R_i$ as the set of vertices of $H_i$ having a neighbor in $W$, that is, $R_i=\{v\in H_i: N(v)\cap W\ne\emptyset\}$, $1\le i\le 6$. Since $G$ is $P_8$-free $R_1=R_6=\emptyset$.\\

Let $w\in W$. We show that there cannot exist three indices $1\le i<j<k\le 6$ such that $w$ has three neighbors $r_i\in R_i$, $r_j\in R_j$, $r_k\in R_k$. Suppose for contradiction that these three neighbors of $w$ exist. Since $R_1=R_6=\emptyset$ then $2\le i<j<k\le 5$.
Since $G$ is claw-free and $H_p$ is anticomplete to $H_{p+1}$, these three indices cannot be successive. So w.l.o.g. we can assume that $i=2$, $j=4$, $k=5$. Now $H_2$ is anticomplete to $H_5$, but $v_3-r_2-w-r_5-v_5-v_4-v_3=C_6$, a contradiction. Hence for every $w\in W$ there is at most two neighbors $r_i,r_j$ such that $r_i\in R_i$, $r_j\in R_j$, $i\ne j$.

If $w$ has two neighbors $r_i\in R_i$, $r_j\in R_j$, $i<j$, then either $r_i\in R_2$, $r_j\in R_4$ or $r_i\in R_3$, $r_j\in R_5$ (recall that $H_i$ is anticomplete to $H_{i+1}$, $H_{p}$, $p \geq i + 3$ and $R_1 = R_6 = \emptyset$).

If $w$ has two neighbors $r_i\in R_2$, $r_j\in R_4$, respectively $r_i\in R_3$, $r_j\in R_5$, then $r_ir_j\in E$, else $w-r_j-v_4-v_3-r_i-w=C_5\subseteq_i G$, respectively $w-r_j-v_5-v_4-r_i-w=C_5\subseteq_i G$, a contradiction. \\

Let $Z_{24}=\{w\in W : N_{R_2}(w)\neq \emptyset,\ N_{R_4}(w)\neq \emptyset\}$ and $Z_{35}=\{w\in W : N_{R_3}(w)\neq \emptyset,\ N_{R_5}(w)\neq \emptyset\}$. We show that $Z_{24}$ is anticomplete to $Z_{35}$. For contradiction we suppose that there are $w_1\in Z_{24}$, $w_2\in Z_{35}$ with $w_1w_2\in E$. Let $r_1\in R_2$ be a neighbor of $w_1$ and $r_2\in R_5$ be a neighbor of $w_2$. Since $r_1r_2\not\in E$ we have $w_1-r_1-v_3-v_4-v_5-v_6-r_2-w_2-w_1=C_8$, a contradiction.

We show that we can assume that $Z_{24}$ and $Z_{35}$ are two independent sets. The two sets being symmetric we show that $Z_{24}$ is an independent set. For contradiction we assume that there are $w_1,w_2\in Z_{24}$ such that $w_1w_2\in E$. We prove that $N_{R_2}(w_1)=N_{R_2}(w_2)$. If $N_{R_2}(w_1) \ne N_{R_2}(w_2)$ then there exists $r_2\in R_2$ which is a neighbor of $w_1$ but not a neighbor of $w_2$. Then $w_2-w_1-r_2-v_3-\cdots-v_7=P_8$, a contradiction. We prove that $N_{R_4}(w_1) = N_{R_4}(w_2)$. If $N_{R_4}(w_1) \ne N_{R_4}(w_2)$ then there exists $r_4\in R_4$ which is a neighbor of $w_1$ but not a neighbor of $w_2$. There exists $r_2\in R_2$ a neighbor of $w_1$ and $w_2$. We know that $r_2r_4 \in E$. It follows that $G[\{v_2,r_2,r_4,w_2\}]$ is a claw, a contradiction. Hence $N_{R_2}(w_1) = N_{R_2}(w_2)$ and $N_{R_4}(w_1) = N_{R_4}(w_2)$. By Property \ref{contract} there exists $s \not\in R_2 \cup R_4$ such that $s$ is a neighbor of $w_1$ but not a neighbor of $w_2$. Let $r_2 \in R_2$ a neighbor of $w_1$ and $w_2$. If $sr_2\not\in E$ then $s-w_1-r_2-v_3-\cdots-v_7=P_8$, a contradiction. When $sr_2\in E$ then $G[\{v_2,r_2,s,w_2\}]$ is a claw, a contradiction. Hence $Z_{24}$ is an independent and by symmetry $Z_{35}$ is also independent.
Moreover, since $G$ is claw-free for every two distinct $w,w'\in Z_{24}\cup Z_{35}$ we have $N(w) \cap N(w')=\emptyset$.

For every two distinct $w,w'\in Z_{24}$, respectively $w,w'\in Z_{35}$ we have that $N_{R_2}(w)$ is anticomplete to $N_{R_4}(w')$ and $N_{R_4}(w)$ is anticomplete to $N_{R_2}(w')$, respectively $N_{R_3}(w)$ is anticomplete to $N_{R_5}(w')$ and $N_{R_5}(w)$ is anticomplete to $N_{R_3}(w')$. For contradiction we assume that $w$ has a neighbor $r_2\in R_2$, $w'$ has a neighbor $r_4\in R_4$, and $r_2r_4\in E$. Then $G[\{v_2,r_2,w,r_4\}]$ is a claw, a contradiction.\\

Let $Z_i=\{w\in W: N(w)\cap R_i\setminus(N_{R_i}(Z_{24}\cup Z_{35}) \neq \emptyset\}$, $2\le i\le 5$.
We show that $Z_2,Z_3,Z_4,Z_5$ are pairwise anticomplete. If there is an edge $w_2w_4$, $w_2\in Z_2$, $w_4\in Z_4,$ with $r_2'\in R_2$, $r_4'\in R_4$ the neighbors of $w_2,w_4$ respectively, then $w_2-r_2'-v_3-v_4-r_4'-w_4-w_2=C_6$ ($r_2'r_4'\not \in E$ else $G[\{v_2,r_2',w_2,r_4'\}]$ is a claw), a contradiction. By symmetry there is no edges between $Z_3,Z_5$. If there is an edge $w_2w_5,w_2\in Z_2$, $w_5\in Z_5$, with $r_2'\in R_2$, $r_5'\in R_5$ the neighbors of $w_2,w_5$ respectively, then $w_2-r_2'-v_3-v_4-v_5-r_5'-w_5-w_2=C_7$ (remember $r_2'r_5'\not \in E$), a contradiction. If there is an edge $w_4w_5,w_4\in Z_4$, $w_5\in Z_5,$ with $r_4'\in R_4$, $r_5'\in R_5$ the neighbors of $w_4,w_5$ respectively, then $w_4-r_4'-v_5-r_5'-w_5-w_4=C_5$ (recall $r_4'r_5' \not\in E$), a contradiction. By symmetry there is no edges between $Z_2,Z_3$. \\

Let $Y=W\setminus(Z_2\cup Z_3\cup Z_4\cup Z_5\cup Z_{24}\cup Z_{35})$. One can observe that for every $w\in Y$ we have $N_{Z_2}(w) = N_{Z_5}(w) = N_{Z_{24}}(w)=N_{Z_{35}}(w)= \emptyset$ else $P_8\subseteq_i G$. Now, if $w\in Y$ has two neighbors $w_3\in Z_3$, $w_4\in Z_4$ then $C_6\subseteq_i G$, a contradiction.
Hence $Y=Y_3\cup Y_4$ with $Y_3\cap Y_4=\emptyset$, $Y_3=\{w\in Y: N_{Z_3}(w)\neq \emptyset\}$, $Y_4=\{w\in Y: N_{Z_4}(w)\neq \emptyset\}$. Moreover $Y_3$ is anticomplete to $Y_4$ else $C_7\subseteq_i G$.\\

We show that we can assume that $Z_2,Z_5,Y_3,Y_4$ are four independent sets. The arguments are the same for the four sets, so we show that the statement is true for $Z_2$. For contradiction we assume that there are $w_1,w_2\in Z_2$ such that $w_1w_2\in E$.
We prove that $N_{R_2}(w_1) = N_{R_2}(w_2)$. If $N_{R_2}(w_1) \ne N_{R_2}(w_2)$ then there exists $r_2\in R_2$ which is a neighbor of $w_1$ but not a neighbor of $w_2$. Then $w_2-w_1-r_2-v_3-\cdots-v_7=P_8$, a contradiction. Since $N(w_1),N(w_2)\subseteq Z_2\cup R_2$ the result is obtained by Property \ref{contract}. Hence $Z_2,Z_5,Y_3,Y_4$ are four independent sets.

Since $G$ is claw-free then for every two distinct vertices $w_1,w_2 \in Z_2\cup Z_5\cup Y_3\cup Y_4$ we have $N(w_1)\cap N(w_2)=\emptyset$.

We prove that for every $w\in Y_3\cup Y_4$, $N(w)$ is a clique. The two cases being symmetric, let $w\in Y_4$. Suppose that there are $s,s'$ two non adjacent vertices in $N(w)$. Since $G$ is claw-free, $s,s'$ cannot have a common neighbor in $R_4$. Let $r\in R_4$ be a neighbor of $s$. Then $s'-w-s-r-v_4-v_3-v_2-v_1=P_8$, a contradiction. \\

Since $G$ is claw-free, if there is a vertex $r \in R_i$ with a neighbor $z \in Z_i$ and a vertex $s \in S$ such as $sz \not\in E$ and $v_i \not\in N(s)$ then $G$ contains a claw (note that $v_{i+1} \not\in N(s)$ is symmetric). Hence $N(Z_i)$ is anticomplete to $H_j$, $j \neq i$. \\

We show that we can assume that $Z_2=Z_5=\emptyset$. The arguments are the same in the two cases, so we consider $Z_2$.
Let $r,r'\in R_2$ be two neighbors of $w\in Z_2$. By previous arguments, $N(w)$ is complete to $H_2$ but anticomplete to $H_1, H_3, H_4, H_5, H_6$. Hence, it remains the case where $N_{S_3 \cup S_4}(r) \neq N_{S_3 \cup S_4}(r')$. Suppose that $N_{S_3 \cup S_4}(r) \neq N_{S_3 \cup S_4}(r')$. Let $s \in S_3\cup S_4$ such as $rs \in E$ but $r's \not\in E$. If $\{v_2,v_3\}\not\subset N_P(s)$ then $G[\{r,s,v_2,w\}]$ or $G[\{r,s,v_3,w\}]$ is a claw, a contradiction. So $\{v_2,v_3\} \subset N_P(s)$. If $v_1 \not\in N_P(s)$, respectively $v_4 \not\in N_P(s)$, then $G[\{r',s,v_1,v_2\}]$, respectively $G[\{r',s,v_3,v_4\}]$, is a claw, a contradiction. Hence $N_P(s) = \{v_1,v_2,v_3,v_4\}$ but $G[\{r,s,v_1,v_4\}]$ is a claw, a contradiction. Hence, $N[r]=N[r']$, a contradiction. Then $w$ is a leaf , a contradiction.\\

Now we study the structures of $Z_3$ and $Z_4$. The two cases being symmetric we deal with $Z_4$. For every distinct vertices $w_1,w_2\in Z_4$ such that $w_1w_2\in E$, then there cannot exist two distinct vertices $w_1',w_2'\in Z_4$ such that $w_1'w_1\in E,w'_1w_2\not\in E$ and $w_2'w_2\in E$, $w'_2w_1\not\in E$. For contradiction we suppose that such two vertices exist. First, we suppose that $w_1, w_2$ have two distinct neighbors $r_1,r_2\in R_4$, respectively. If $w_1'r_2\not\in E$ then $v_1-v_2-v_3-v_4-r_2-w_2-w_1-w_1'=P_8$, a contradiction. If $w_1'r_2\in E$ then $G[\{v_4,r_2,w_2,w_1'\}]$ is a claw, a contradiction. Second, w.l.o.g., $r_1\in R_4$ is a common neighbor of $w_1,w_2$ and $r_2\in R_4$ is a neighbor of $w_2$ but not $w_1$. If $w_1'r_2\not\in E$ then $v_1-v_2-v_3-v_4-r_2-w_2-w_1-w_1'=P_8$ else $G[\{v_4,r_2,w_2,w_1'\}]$ is a claw, a contradiction. Finally, $r_1,r_2\in R_4$ are two common neighbors of $w_1,w_2$ ($r_1,r_2$ are not necessarily distinct). If, w.l.o.g., $w'_1r_1\in E$ then $G[\{v_4,r_1,w_2,w_1'\}]$ is a claw, a contradiction. In the case where $w'_1r_1,w_1'r_2\not\in E$ then $w_1'$ has a neighbor $r'_1\in R_4$, $r'_1\ne r_1,r_2$. If $r'_1w_2\in E$ then $G[\{v_4,r_1',w_2,w_1'\}]$ is a claw, a contradiction. So $r'_1w_2\not\in E$. If $r'_1w_1\not\in E$ then $v_1-v_2-v_3-v_4-r_1'-w_1'-w_1-w_2=P_8$, a contradiction. Thus $r'_1w_1\in E$. If $r'_1w_2'\not\in E$ then $v_1-v_2-v_3-v_4-r_1'-w_1-w_2-w_2'=P_8$, a contradiction. So $r'_1w_2'\in E$ but $G[\{v_4,r_1',w_1,w_2'\}]$ is a claw, a contradiction.

As a consequence each connected component $A_i$ of $Z_3 \cup Z_4$ has a universal vertex. Also, $G$ being claw-free two distinct components cannot share a neighbor in $R_3 \cup R_4$. Moreover each $w_4\in Z_3 \cup Z_4 $ is not a leaf. \\

We show that $w\in Y_3 \cup Y_4$ is connected to a universal vertex of a connected component $A_i$ of $Z_3 \cup Z_4$. The two cases being symmetric, we deal with $Z_4$. Let $w \in Z_4$. We assume that the neighbors of $w$ are not universal in $A_i$. Let $s\in A_i$ be a neighbor of $w$, let $u$, $u\ne s$, be a universal vertex of $A_i$. Since $s$ is not universal there exists $v$, $v\in A_i$ such that $sv\not\in E$ and $uv\in E$. Since $N(w)$ is complete $wv\not\in E$.
Let $r\in R_4$ be a neighbor of $s$. Since $G$ is claw-free then $rv\not\in E$. Let $r'$, $r'\in R_4$, $r'\ne r,$ be a neighbor of $v$. As just above $r's\not\in E$. If $r'u\not\in E$ then $v_1-v_2-v_3-v_4-r'-v-u-s=P_8$ else $v_1-v_2-v_3-v_4-r'-u-s-w=P_8$, a contradiction. \\

We are ready to show how to build a $\gamma$-set in polynomial time. \\

First, we treat the case where $Z_{24}\ne \emptyset$ (the case $Z_{35}\ne \emptyset$ is the same). Let $r_2\in R_2$,$ r_4\in R_4$ be the two neighbors of $w$, $w\in Z_{24}$. Recall that $N(Z_{24})\subseteq R_2\cup R_4$.

We show that $R_3=\emptyset$. Assume that there exists $w'\in W$ with a neighbor $r_3\in H_3$ (thus $R_3\ne\emptyset$). Note that $w'$ cannot be a neighbor of $r_2$ or $r_4$. Then $w'-r_3-v_3-r_2-r_4-v_5-v_6-v_7=P_8$, a contradiction. An immediate consequence is that $Z_{35}=\emptyset$.
There is no vertex $w', w'\in W,$ with $r_2$ as a neighbor else $G[\{v_2,r_2,r_4,w'\}]$ is a claw. By symmetry, there is no vertex $w', w'\in W,$ with $r_4$ as a neighbor.\\

Let $r_2\in R_2$, $r_2'\in R_2$, $r_2\ne r_2'$ be such that $r_2$, respectively $r_2'$, has a neighbor $w\in Z_{24}$, respectively $w'\in Z_{24}$.
Let $r_4\in R_4$, $r_4'\in R_4$, $r_4\ne r_4'$ be such that $r_4$, respectively $r_4'$, has $w$, respectively $w'$, as neighbor.
We show that $N_{S\setminus H_4}(r_2)=N_{S\setminus H_4}(r'_2)$, respectively $N_{S\setminus H_2}(r_4)=N_{S\setminus H_2}(r'_4)$.

Let $i=2$ (the case $i=4$ is symmetric). By contradiction, we assume that there exists $s\in S\setminus H_4$ such that $r_2s\in E$, $r'_2s\not\in E$. From above $s\not\in S_2$. So $s\in S_3\cup S_4$. If $N_P(s)=\{v_1,v_2,v_3\}$ then $G[\{r'_2,v_3,v_4,s\}]$ is a claw, a contradiction.
When $N_P(s)=\{v_2,v_3,v_4\}$ then $G[\{r'_2,v_1,v_2,s\}]$ is a claw, a contradiction. When $N_P(s)=\{v_j,v_{j+1},v_{j+2}\},j\ge 3,$ then $G[\{r_2,w,v_2,s\}]$ is a claw, a contradiction. Thus $s\in S_4$. When $N_P(s)=\{v_1,v_2,v_3,v_4\}$ then $G[\{r_2,v_1,v_4,s\}]$ is a claw, a contradiction. When $N_P(s)=\{v_1,v_2,v_j,v_{j+1}\},j\ge 4,$ then $G[\{r_2,v_1,v_j,s\}]$ is a claw, a contradiction. When $N_P(s)=\{v_2,v_3,v_j,v_{j+1}\}$,$j\ge 4,$ then $G[\{r_2',v_1,v_2,s\}]$ is a claw, a contradiction. When $N_P(s)=\{v_j,v_{j+1},v_{j'},v_{j'+1}\},j\ge 3,j'\ge j+2$ then $G[\{r_2,w,v_2,s\}]$ is a claw, a contradiction. Hence $N_{S\setminus H_4}(r_2)=N_{S\setminus H_4}(r'_2)$ and $N_{S\setminus H_2}(r_4)=N_{S\setminus H_2}(r'_4)$.  \\

Let $r_4\in R_4$, $r_4'\in R_4$, $r_4\ne r_4'$ be such that $r_4$, respectively $r_4'$, has a neighbor $w\in Z_4$, respectively $w'\in Z_4$.
We show that $N_S(r_4)=N_S(r'_4)$.

By contradiction, we assume that there exists $s\in S$ such that $r_4s\in E$, $r'_4s\not\in E$. From above $s\not\in H_1\cup H_3\cup H_4 \cup H_5$. So $s\in H_2\cup H_6\cup S_3\cup S_4$. If $s\in H_2$ or $s\in H_6$ then $G[\{v_4,w,r_4,s\}]$ is a claw, a contradiction. So $s\in S_3\cup S_4$. If  $N_P(s)=\{v_1,v_2,v_3\}$ then $G[\{v_1,v_3,r_4,s\}]$ is a claw, a contradiction. If  $N_P(s)=\{v_2,v_3,v_4\}$ then $G[\{v_5,w,r_4,s\}]$ is a claw, a contradiction.
If $N_P(s)=\{v_3,v_4,v_5\}$ then $G[\{v_5,v_6,r'_4,s\}]$ is a claw, a contradiction. If $N_P(s)=\{v_4,v_5,v_6\}$ then $G[\{v_3,v_4,r'_4,s\}]$ is a claw, a contradiction.
When $N_P(s)=\{v_5,v_6,v_7\}$ then $G[\{v_4,w,r_4,s\}]$ is a claw, a contradiction. Thus $s\in S_4$. When $N_P(s)=\{v_1,v_2,v_3,v_4\}$ then $G[\{r_4,v_1,v_3,s\}]$ is a claw, a contradiction. When $N_P(s)=\{v_2,v_3,v_4,v_5\}$ then $G[\{r_4',v_5,v_6,s\}]$ is a claw, a contradiction. When $N_P(s)=\{v_3,v_4,v_5,v_6\}$ then $G[\{r_4,v_3,v_6,s\}]$ is a claw, a contradiction. When $N_P(s)=\{v_4,v_5,v_6,v_7\}$ then $G[\{r'_4,v_3,v_4,s\}]$ is a claw, a contradiction.
When $N_P(s)=\{v_1,v_2,v_6,v_7\}$ or $N_P(s)=\{v_2,v_3,v_6,v_7\}$ then $G[\{r_4,v_2,v_6,s\}]$ is a claw, a contradiction. When $N_P(s)=\{v_3,v_4,v_6,v_7\}$ then $G[\{r_4,v_3,v_6,s\}]$ is a claw, a contradiction. When $N_P(s)=\{v_1,v_2,v_5,v_6\}$ or $N_P(s)=\{v_2,v_3,v_5,v_6\}$ then $G[\{r_4,v_2,v_6,s\}]$ is a claw, a contradiction. When $N_P(s)=\{v_1,v_2,v_4,v_5\}$ then $G[\{r'_4,v_5,v_6,s\}]$ is a claw, a contradiction. Hence $N_S(r_4)=N_S(r'_4)$. By symmetry, for $r_2\in R_2$, $r_2'\in R_2$, $r_2\ne r_2'$ such that $r_2$, respectively $r_2'$, has a neighbor $w\in Z_2$, respectively $w'\in Z_2$ we have $N_S(r_2)=N_S(r'_2)$.\\

The $\gamma$-set is build as follows:

\begin{itemize}
\item $\vert Z_{24}\vert\ge 2$. We take $r_2 \in R_2$ a neighbor of $w$, and for each other $w'\in Z_{24}$ we take one adjacent vertex $r'_4\in R_4$. For each $w'\in Y_4$ we take one universal vertex in the connected component $A_i$ of $Z_4$ connected to $w'$. For each connected component $A_i$ of $Z_4$ that is not connected with $Y_4$, we take one universal vertex of $A_i$. These vertices dominate $ Z_{24}\cup Y_4\cup Z_4\cup H_2\cup H_4\cup\{v_2,v_3,v_4,v_5\}$. Since $v_1,v_7$ have no common neighbor at least two more vertices are needed. Adding the three vertices $v_2,v_4,v_6$ we have a dominating set (not necessarily minimum). Checking for all the pairs $s_1,s_7$ where $s_i$ is a neighbor of $v_i$, $i\in \{1,7\}$, one can verify if there is a $\gamma$-set with only two more vertices (note that there are at most $O(n^2)$ such pairs).

\item $\vert Z_{24}\vert=1$. For each $w' \in Y_4$ we take one universal vertex in the connected component $A_i$ of $Z_4$ connected to $w'$. If there exists a vertex $r \in R_4$ complete to a component $A_i$ of $Z_4$ that is not connected to $Y_4$ then we take $r$. For each remaining component $A_i$ of $Z_4$ that is not connected to $Y_4$, we take one universal vertex of $A_i$. These vertices dominate $Y_4\cup Z_4$ (note that $H_2, H_4$ are not necessarily dominated). Since $v_1,v_7,w$ have no common neighbor at least three more vertices are needed. Adding the four vertices $v_2,v_4,v_6,w$ we have a dominating set (not necessarily minimum). Checking for all the pairs $s_1,s_7$ where $s_i$ is a neighbor of $v_i,i\in \{1,7\}$, if there is a dominating set by adding $s_1,s_7,r_4$ or $s_1,s_7,r_2$, one can verify if there is a $\gamma$-set with only three more vertices (note that there are at most $O(n^2)$ such pairs).
\end{itemize}

In the case of $Z_{24}=Z_{35}=\emptyset$, we build the $\gamma$-set as follows:
\begin{itemize}
\item $Y_3,Y_4\ne\emptyset$. For each $w \in Y_3 \cup Y_4$ we take one universal vertex in the connected component $A_i$ of $Z_3\cup Z_4$ connected to $w$. If there exists $r_4\in R_4$ which is complete to a component $A_i$ of $Z_4$ that is not connected to $Y_3 \cup Y_4$ then, we take $r_4$. We do the same for the component of $Z_3$ with no neighbors in $Y_3$. For each remaining connected component $A_i$ of $Z_3\cup Z_4$ that is not connected to $Y_3 \cup Y_4$, we take one universal vertex of $A_i$. These vertices dominate $Y_3\cup Z_3\cup Y_4\cup Z_4$ (note that $H_2, H_4$ are not necessarily dominated). Since $v_1,v_7$ have no common neighbor at least two more vertices are needed. Adding the three vertices $v_2,v_4,v_6$ we have a dominating set (not necessarily minimum). Checking for all the pairs $s_1,s_7$ where $s_i$ is a neighbor of $v_i,i\in \{1,7\}$, one can verify if there is a $\gamma$-set with only two more vertices (note that there are at most $O(n^2)$ such pairs).
\item $Y_3\ne\emptyset$, $Y_4=\emptyset$ or $Y_4 \neq\emptyset$, $Y_3 = \emptyset$. The two cases being symmetric, let $Y_4=\emptyset$.

\begin{itemize}
\item $Z_4\ne\emptyset$. For each $w \in Y_3$ we take one universal vertex in the connected component $A_i$ of $Z_3$ connected to $w$. If there exists $r_4\in R_4$ which is complete to $A_i$, a connected component of $Z_4$, then we take $r_4$. If there exists $r_3 \in R_3$ which is complete to a connected component $A_j$ of $Z_3$ with no neighbors in $Y_3$, then we take $r_3$. Now, we take one universal vertex for each other component $A_l$, $A_l \neq A_i, A_j,$ of $Z_3\cup Z_4$. These vertices dominate $Y_3\cup Z_3\cup Z_4$. Since $v_1,v_7$ have no common neighbor at least two more vertices are needed. Adding the three vertices $v_2,v_4,v_6$ we have a dominating set (not necessarily minimum). Checking for all the pairs $s_1,s_7$ where $s_i$ is a neighbor of $v_i,i\in \{1,7\}$, one can verify if there is a $\gamma$-set with only two more vertices (note that there are at most $O(n^2)$ such pairs).

\item $Z_4=\emptyset$. For each $w \in Y_3$ we take one universal vertex in the connected component $A_i$ of $Z_3$ connected to $w$. If there exists $r_3 \in R_3$ which is complete to a connected component $A_i$ of $Z_3$ with no neighbors in $Y_3$, then we take $r_3$. Now, we take one universal vertex for each other component $A_i$ of $Z_3$. Adding the vertices $v_2,v_4,v_6$ we have a dominating set (not necessarily minimum). Checking for all the pairs $s_1,s_7$ where $s_i$ is a neighbor of $v_i,i\in \{1,7\}$, one can verify if there is a $\gamma$-set with only two more vertices. \end{itemize}

\item $Y_3,Y_4=\emptyset$.

\begin{itemize}
\item $Z_3,Z_4\ne\emptyset$. If there exists $r_4\in R_4$, respectively $r_3\in R_3$, which is complete to $A_i$, a connected component of $Z_4$, respectively $Z_3$, then we take $r_4$, respectively $r_3$. For each remaining component of $Z_3\cup Z_4$ we take one universal vertex. Adding the vertices $v_2,v_4,v_6$ we have a dominating set (not necessarily minimum). Checking for all the pairs $s_1,s_7$ where $s_i$ is a neighbor of $v_i,i\in \{1,7\}$, one can verify if there is a $\gamma$-set with only two more vertices.
\item $Z_3\ne\emptyset$, $Z_4=\emptyset$ or $Z_4\ne\emptyset$, $Z_3=\emptyset$. Let $Z_3\ne\emptyset$. If there exists $r_3\in R_3$ which is complete to a connected component of $Z_3$, then we take $r_3$. We add one universal vertex for each remaining component of $Z_3$. Now, adding the vertices $v_2,v_4,v_6$ we have a dominating set (not necessarily minimum). Checking for all the pairs $s_1,s_7$ where $s_i$ is a neighbor of $v_i,i\in \{1,7\}$, one can verify if there is a $\gamma$-set with only two more vertices.
\item $Z_3=Z_4=\emptyset$. Then $V=N[V(C)]$ and by Property \ref{V=N[T]} computing a minimum dominating set is polynomial.
\end{itemize}
\end{itemize}
\end{proof}

From Lemmas \ref{Ck}, \ref{Ck-1}, \ref{C78-2}, \ref{clawp8}, \ref{P8C}, we obtain the main result of this paper.

\begin{theorem}
The Minimum Dominating Set problem is polynomial for $(claw,P_8)$-free graphs.
\end{theorem}
\section{Conclusion}

We have shown that the Minimum Dominating Set problem is polynomial for $(claw,P_8)$-free graphs. We left open the following problem: is there a positive integer $k, k\ge 9,$ such that the Minimum Dominating Set problem is $NP$-complete for the class of $(claw,P_k)$-free graphs? If the the answer is positive, a challenge should be to show a dichotomy: find the minimum integer $k$ such that the Minimum Dominating Set problem is $NP$-complete for $(claw,P_k)$-free graphs and polynomial for $(claw,P_{k-1})$-free graphs.

\end{document}